\documentclass[a4paper]{article}

\usepackage{amsmath,amssymb}
\usepackage{amscd}
\usepackage{mathrsfs}
\usepackage{amsthm}
\usepackage[all]{xy}
\usepackage{color}

\theoremstyle{plain}

\newtheorem{theorem}{\bf Theorem}[section]
\newtheorem{lemma}[theorem]{\bf Lemma}
\newtheorem{proposition}[theorem]{\bf Proposition}
\newtheorem{corollary}[theorem]{\bf Corollary}

\theoremstyle{definition}
\newtheorem{definition}[theorem]{\bf Definition}
\newtheorem{example}[theorem]{\bf Example}
\newtheorem{remark}[theorem]{\bf Remark}

\newtheorem{question}[theorem]{\bf {Question}}

\newcommand{\eqa}[1]{
\begin{align*}
#1
\end{align*}}

\newcommand{\nai}[2]{\langle #1,#2\rangle}
\newcommand{\dom}[1]{{{\rm{dom}}{(#1)}}}

  \makeatletter
  \newcommand{\subsubsubsection}{\@startsection{paragraph}{4}{\z@}%
    {1.0\Cvs \@plus.5\Cdp \@minus.2\Cdp}%
    {.1\Cvs \@plus.3\Cdp}%
    {\reset@font\sffamily\normalsize}
  }
  \makeatother
\title{Weyl-von Neumann Theorem and Borel Complexity of Unitary Equivalence Modulo Compacts of Self-Adjoint Operators}

\author{Hiroshi Ando \and Yasumichi Matsuzawa}

\begin{document}
\maketitle

\begin{abstract}
Weyl-von Neumann Theorem asserts that two bounded self-adjoint operators $A,B$ on a Hilbert space $H$ are unitarily equivalent modulo compacts, i.e., $uAu^*+K=B$ for some unitary $u\in \mathcal{U}(H)$ and compact self-adjoint operator $K$, if and only if $A$ and $B$ have the same essential spectra: $\sigma_{\rm{ess}}(A)=\sigma_{\rm{ess}}(B)$. 
In this paper we consider to what extent the above Weyl-von Neumann's result can(not) be extended to unbounded operators using descriptive set theory. We show that if $H$ is separable infinite-dimensional, this equivalence relation for bounded self-adjoin operators is smooth, while  the same equivalence relation for general self-adjoint operators  contains a dense $G_{\delta}$-orbit but does not admit classification by countable structures. On the other hand, apparently related equivalence relation $A\sim B\Leftrightarrow \exists u\in \mathcal{U}(H)\   [u(A-i)^{-1}u^*-(B-i)^{-1}$ is compact], is shown to be smooth. 
Various Borel or co-analytic equivalence relations related to self-adjoint operators are also presented. 
\end{abstract}


\medskip

\noindent
\tableofcontents
\medskip

\section{Introduction}
The celebrated Weyl-von Neumann Theorem \cite{Weyl,von Neumann} asserts that any bounded self-adjoint operator can be turned into a diagonalizable operator with arbitrarily small compact perturbations. More precisely:
\begin{theorem}[Weyl-von Neumann]\label{thm: Weyl-von Neumann}
Let $A$ be a (not necessarily bounded) self-adjoint operator on a separable Hilbert space $H$ and $\varepsilon>0$, there exists a compact operator $K$ with $\|K\|<\varepsilon$, such that $A+K$ is of the form \[A+K=\sum_{n=1}^{\infty}a_n\nai{\xi_n}{\ \cdot\ }\xi_n,\]
 where $a_n\in \mathbb{R}$ and $\{\xi_n\}_{n=1}^{\infty}$ is a CONS for $H$.
\end{theorem}
Weyl obtained Theorem \ref{thm: Weyl-von Neumann} for bounded operators without norm estimates on $K$, and the present form of the theorem was obtained by von Neumann. Moreover, he also proved that the $K$ can be chosen to be of Hilbert-Schmidt class (in fact $K$ can be chosen to be Schatten $p$-class for any $p>1$ by \cite{Kuroda58}, but $p=1$ is impossible by \cite{Kato57,Rosenblum}. See \cite{Conway,Putnam,ReedSimonI,AkhiezerGlazman} for details). 
Berg \cite{Berg} generalized Theorem \ref{thm: Weyl-von Neumann} to (unbounded) normal operators. 

On the other hand, Weyl \cite{Weyl10} proved that the essential spectra of a self-adjoint operator is invariant under compact perturbations. 
Here, the essential spectra $\sigma_{\rm{ess}}(A)$ of a self-adjoint operator $A$ is the set of all $\lambda$ in the spectral set $\sigma(A)$ of $A$ which is either an eigenvalue of infinite multiplicity or an accumulation point in $\sigma(A)$. 
Based on Theorem \ref{thm: Weyl-von Neumann}, von Neumann showed ((1)$\Rightarrow $(2) below) that up to unitary conjugation, the converse to Weyl's compact perturbation Theorem holds:
\begin{theorem}[Weyl-von Neumann]\label{thm: von Neumann theorem}
Let $A,B$ be bounded self-adjoint operators on $H$. Then the following conditions are equivalent:
\begin{list}{}{}
\item[{\rm{(1)}}] $\sigma_{\rm{ess}}(A)=\sigma_{\rm{ess}}(B)$. 
\item[{\rm{(2)}}] $A$ and $B$ are unitarily equivalent modulo compacts. More precisely, there exists a compact self-adjoint operator $K$ on $H$ and a unitary operator $u$ on $H$, such that 
\[uAu^*+K=B.\]
\end{list}
\end{theorem} 
Theorem \ref{thm: von Neumann theorem} states that the essential spectra is a complete invariant for the classification problem of all bounded self-adjoint operators up to unitary equivalence modulo compacts. 
On the other hand, Theorem \ref{thm: Weyl-von Neumann} and Weyl's Theorem \ref{thm: von Neumann theorem} (2)$\Rightarrow $(1) above also holds for unbounded self-adjoint operators. It is therefore of interest to know whether Theorem \ref{thm: von Neumann theorem} holds true for general unbounded self-adjoint operators. However, a simple example (Example \ref{ex: ess spec=0 counterexample}) 
clarifies that von Neumann's Theorem \ref{thm: von Neumann theorem} (1)$\Rightarrow $(2) cannot be generalized verbatim for unbounded operators. Moreover, further examples (Examples \ref{ex: non-unitarily equivalent but have same ess spec} and \ref{ex: unitarily quivalent domains but not uAu*+K=B})  show that it seems impossible to find a reasonable complete invariant characterizing this equivalence which is assigned to each self-adjoint operators in a constructible way. 

It is the purpose of the present paper to show that there is a sharp contrast between the complexity of the above classification problem for bounded operators and that for unbounded operators by descriptive set theoretical method, especially the turbulence theorem established by Hjorth \cite{Hjorth00}. 
More precisely, we prove the following: let $H$ be a separable infinite-dimensional Hilbert space, and ${\rm{SA}}(H)$ be the Polish space of all (possibly unbounded) self-adjoint operators equipped with the strong resolvent topology (SRT, see $\S$\ref{sec: SA(H) is Polish}). Then the set $\mathbb{B}(H)_{\rm{sa}}$ of bounded self-adjoint operators on $H$ is a Borel subset of ${\rm{SA}}(H)$ (Lemma \ref{lem: B(H)_sa is standard}). Consider the semidirect product Polish group $G=\mathbb{K}(H)_{\rm{sa}}\rtimes \mathcal{U}(H)$, where $\mathbb{K}(H)_{\rm{sa}}$ is the additive Polish group of compact self-adjoint operators with the norm topology, and we equip the unitary group $\mathcal{U}(H)$ of $H$ with the strong operator topology. The action of $\mathcal{U}(H)$ on $\mathbb{K}(H)_{\rm{sa}}$ is given by conjugation. Then we consider the orbit equivalence relation $E_G^{{\rm{SA}}(H)}$ of the $G$-action on ${\rm{SA}}(H)$ given by $(K,u)\cdot A:=uAu^*+K\  (u\in \mathcal{U}(H),K\in \mathbb{K}(H)_{\rm{sa}})$. Since $\mathbb{B}(H)_{\rm{sa}}$ is a $G$-invariant Borel subset , and we may consider the restricted equivalence relation $E_G^{\mathbb{B}(H)_{\rm{sa}}}$ as well. Therefore, the difference of the complexity of the above classification for bounded vs unbounded operators should be understood as the difference of the complexities of $E_G^{{\rm{SA}}(H)}$ and $E_G^{\mathbb{B}(H)_{\rm{sa}}}$. In this respect, let us now state our main theorem (if the reader is not familiar with operator theory or Borel equivalence relations, basic facts are summarized in $\S$\ref{sec: preliminaries} and all the necessary notions are defined there):
\begin{theorem}
Denote by $\mathcal{F}(\mathbb{R})$ the Effros Borel space of closed subsets of $\mathbb{R}$ (see $\S$\ref{subsec: Effros Borel space}). 
The following statements hold:
\begin{list}{}{}
\item[{\rm{(1)}}] ${\rm{SA}}(H)\ni A\mapsto \sigma_{\rm{ess}}(A)\in \mathcal{F}(\mathbb{R})$ is Borel.
In particular, $E_G^{\mathbb{B}(H)_{\rm{sa}}}$ is smooth. 
\item[{\rm{(2)}}] There exists a dense $G_{\delta}$ orbit of the $G$-action on ${\rm{SA}}(H)$. In particular, the action is not generically turbulent.  
\item[{\rm{(3)}}] $E_G^{{\rm{SA}}(H)}$ does not admit classification by countable structures.
\end{list}
\end{theorem}
Proofs of (1), (2) and (3) are given in Theorem \ref{thm: taking ess spectrum is Borel}, Theorem \ref{thm: unbounded case: comeager orbit exists} and  Theorem \ref{thm: E_G^SA is unclassifiable by countable structures}, respectively.  In the proof of (1), Christensen's Theorem \cite{Christensen71} asserting the Borelness of $\mathcal{F}(\mathbb{R})\times \mathcal{F}(\mathbb{R})\ni (K_1,K_2)\mapsto K_1\cap K_2\in \mathcal{F}(\mathbb{R})$ plays an important role (cf. \cite{DingGao}). 
Regarding (3), we prove more precisely that the subspace ${\rm{EES}}(H)=\{A\in {\rm{SA}}(H);\sigma_{\rm{ess}}(A)=\emptyset\}$, equipped with the norm resolvent topology (NRT, see $\S$\ref{para: EES, NRT is Polish}) is shown to be a Polish $G$-space (with respect to the restricted action), and the $G$-action on ${\rm{EES}}(H)$ is  generically turbulent (Theorem \ref{thm: action is weakly turbulent}). Since $A\mapsto \sigma_{\rm{ess}}(A)$ is constant ($=\emptyset$) on ${\rm{EES}}(H)$, this shows that the essential spectra is very far from a complete invariant even in this small subspace of ${\rm{SA}}(H)$. Since NRT is stronger than SRT, this shows that $E_G^{{\rm{EES}}(H)}$ is Borel reducible (in fact continuously embeddable) to $E_G^{{\rm{SA}}(H)}$, whence (3) holds by Hjorth turbulence Theorem \cite{Hjorth00}. On the other hand, there is a related equivalence relation: define an equivalence relation $E_{\rm{u.c.res}}^{{\rm{SA}}(H)}$ on ${\rm{SA}}(H)$ by 
\[AE_{\rm{u.c.res}}^{{\rm{SA}}(H)}B\Leftrightarrow \exists u\in \mathcal{U}(H)\  \ 
[u(A-i)^{-1}u^*-(B-i)^{-1}\in \mathbb{K}(H)].\]  
$E_G^{{\rm{SA}}(H)}$ is stronger than $E_{\rm{u.c.res}}^{{\rm{SA}}(H)}$ in the sense that $E_G^{{\rm{SA}}(H)}\subset E_{\rm{u.c.res}}^{{\rm{SA}}(H)}$ (Lemma \ref{lem: von Neuman is stronger than E_res.c}), and $ E_{\rm{u.c.res}}^{{\rm{SA}}(H)}$ restricted to $\mathbb{B}(H)_{\rm{sa}}$ agrees with $E_G^{\mathbb{B}(H)_{\rm{sa}}}$ (Lemma \ref{lem: Matsuzawa relation coincides with von Neumann's on B(H)}). Therefore  $E_{\rm{u.c.res}}^{{\rm{SA}}(H)}$ is considered to be another extension of $E_G^{\mathbb{B}(H)_{\rm{sa}}}$  to ${\rm{SA}}(H)$. We show that unlike $E_G^{{\rm{SA}}(H)}$, $E_{\rm{u.c.res}}^{{\rm{SA}}(H)}$ is actually smooth (Theorem \ref{thm: Matsuzawa relation is smooth}), although the essential spectra cannot be a complete invariant (Example \ref{ex: bounded vs unbounded}).  We also consider other various equivalence relations related to operator theory: we show that the domain equivalence relation given by 
\[AE_{\rm{dom}}^{{\rm{SA}}(H)}B\Leftrightarrow \dom{A}=\dom{B},\] is co-analytic (Proposition \ref{prop: E_dom is coanalytic}), and each class is dense and meager (Proposition \ref{prop: Edom class is dense and meager}).  We do not know if it is Borel. On the other hand we show that the unitary equivalence of domains given by \[AE_{\rm{dom},u}^{{\rm{SA}}(H)}B\Leftrightarrow \exists u\in \mathcal{U}(H)\ [u\cdot \dom{A}=\dom{B}],\] is Borel (Proposition \ref{prop: E_{dom,u} is Borel}), thanks to the theory of operator ranges \cite{FillmoreWilliams}. Furthermore, the action of the additive Polish group $\mathbb{K}(H)_{\rm{sa}}$ on ${\rm{SA}}(H)$ by $K\cdot A:=A+K\ (K\in \mathbb{K}(H)_{\rm{sa}},A\in {\rm{SA}}(H))$, is shown to be generically turbulent (Theorem \ref{thm: action of K(H) on SA(H) is turbulent}). Note that this is in contrast to the fact that the action of the larger group $G$ on ${\rm{SA}}(H)$ is not generically turbulent.  

The connection between descriptive set theory and other areas of mathematics such as ergodic theory or operator algebra theory have been proved to be very fruitful (see e.g. \cite{FarahTomsTornquist,KechrisTuckerDrob13,KechrisSofronidis01,KerrLiPichot10,SasykTornquist08}, and references therein). 
However, apart from the pioneering work of Simon \cite{Simon95} (see also \cite{ChokskiNadkarni,Israel04}) for a special class of self-adjoint operators, apparently no descriptive study has been carried out for the space ${\rm{SA}}(H)$. We hope that the present work not only shows the usefulness of the descriptive set theoretical viewpoint but also verifies that the theory of (unbounded) self-adjoint operators gives us rich examples of interesting equivalence relations. 
\section{Preliminaries}\label{sec: preliminaries}
\subsection{Operator Theory}\label{subsec: operator theory}
Here we recall basic notions from spectral theory. Details can be found e.g. in \cite{ReedSimonI,Schmudgen}. 
 Let $H$ be a separable infinite-dimensional Hilbert space. The group of unitary operators on $H$ is denoted $\mathcal{U}(H)$. We denote $\mathbb{B}(H)$ (resp. $\mathbb{B}(H)_{\rm{sa}}$) the space of all bounded (resp. bounded self-adjoint) operators on $H$, and $\mathbb{K}(H)$ (resp. $\mathbb{K}(H)_{\rm{sa}}$) the space of all compact (resp. compact self-adjoint) operators on $H$. 
The convergence of bounded operators with respect to the strong operator topology (SOT for short) is denoted $x_n\stackrel{\text{SOT}}{\to}x$ or $x_n\to x$ (SOT), which means $\|x_n\xi-x\xi\|\stackrel{n\to \infty}{\to}0$ for each $\xi \in H$. $(\mathcal{U}(H),{\rm{SOT}})$ is a Polish group (i.e., a topological group whose topology is Polish. See $\S$\ref{subsec: Borel equivalence relations}). 
The {\it domain} (resp. {\it range}) of a linear operator $A$ is denoted $\dom{A}$ (resp. $\text{Ran}(A)$). 
$A$ is called {\it densely defined} if $\dom{A}$ is dense.  An operator $B$ on $H$ is called an {\it extension} of $A$, denoted $A\subset B$, if $\dom{A}\subset \dom{B}$ and $A\xi=B\xi$ for all $\xi \in \dom{A}$. For a densely defined operator $A$, the {\it adjoint} $A^*$ of $A$ is defined as follows: its domain is 
 \[\dom{A^*}:=\{\eta\in H;\ \dom{A}\ni \xi\mapsto \nai{\eta}{A\xi}\text{ is continuous}\}.\]
 By Riesz representation Theorem, for $\eta \in \dom{A^*}$ there exists a unique vector $\zeta$, such that $\nai{\eta}{A\xi}=\nai{\zeta}{\xi}$ holds for $\xi\in \dom{A}$. We then define $A^*\eta$ to be $\zeta$.  
The {\it graph} of $A$, denoted as $G(A)$ is the subspace $\{(\xi, A\xi);\ \xi \in \dom{A}\}$ of $H\oplus H$. In $\dom{A}$, we define the {\it graph-norm} of $A$ by
\[\|\xi\|_A:=\sqrt{\|\xi\|^2+\|A\xi\|^2},\ \ \ \xi \in \dom{A}.\]

\begin{definition}Let $A$ be a densely defined operator on $H$.
\begin{list}{}{}
\item[(1)] $A$ is called {\it closed} if the graph $G(A)$ is closed in $H \oplus H$. This is equivalent to say that if $\{\xi_n\}_{n=1}^{\infty}\subset \dom{A}$ converges to $\xi\in H$ and $\{A\xi_n\}_{n=1}^{\infty}$ converges to $\eta$, then $\xi\in \dom{A}$ and $A\xi=\eta$ holds. 
\item[(2)] $A$ is called {\it closable} if it admits an extension $B$ as a closed operator. 
This is equivalent to say that if $\{\xi_n\}_{n=1}^{\infty}\subset \dom{A}$ converges to 0 and $\{A\xi_n\}_{n=1}^{\infty}$ converges to $\eta \in H$, then $\eta=0$. $B$ is called a {\it closed extension} of $A$.
\end{list}
If an operator $A$ is closable, it has the smallest closed extension $\overline{A}$, called the {\it closure} of $A$:
\eqa{
\dom{\overline{A}}&:=\left \{ \xi \in H; \ \exists \{\xi_n\}_{n=1}^{\infty}\subset \dom{A},\ \xi_n\stackrel{n\to \infty}{\to} \xi \text{ and } \lim_{n\to \infty}T\xi_n\text{ exists }\right \}, \\
\overline{A}\xi&:=\lim_{n\to \infty}A\xi_n, \ \ \xi \in \dom{\overline{A}}.}
$A$ is closable if and only if $A^*$ is densely defined, in which case $\overline{A}=(A^*)^*$ holds.
\begin{list}{}{}
\item[(3)] $A$ is called {\it symmetric}, if $A\subset A^*$ i.e., $\nai{A\xi}{\eta}=\nai{\xi}{A\eta}$ holds for $\xi,\eta\in \dom{A}$. 
\item[(4)] $A$ is called {\it self-adjoint}, if $A=A^*$. 
\end{list}
\end{definition}
\begin{definition}
We define ${\rm{SA}}(H)$ to be the space of all (possibly unbounded) self-adjoint operators on $H$. The {\it strong resolvent topology} (SRT for short) is the weakest topology on ${\rm{SA}}(H)$ for which ${\rm{SA}}(H)\ni A\mapsto (A-i)^{-1}\xi\in H$ is continuous for every $\xi \in H$. 
\end{definition}
In other words, a sequence $\{A_n\}_{n=1}^{\infty}$ in ${\rm{SA}}(H)$ converges to $A\in {\rm{SA}}(H)$ in SRT, if and only if ${\rm{SOT}}-\lim_{n\to \infty}(A_n-i)^{-1}=(A-i)^{-1}$. It will be shown (Proposition \ref{prop: SA(H) is Polish}) that ${\rm{SA}}(H)$ equipped with SRT is Polish (see $\S$\ref{subsec: Borel equivalence relations}). 
Let us remark an important difference between $\mathbb{B}(H)_{\rm{sa}}$ and ${\rm{SA}}(H)$: the latter  is {\it not} a vector space. Recall that the {\it sum} and the {\it multiplication} of two densely defined operators $A,B$ on $H$ are defined as 
\eqa{
\dom{A+B}&=\dom{A}\cap \dom{B},\ \ \ \ \ \ \ \ \ \ \ \ \ (A+B)\xi:=A\xi+B\xi,\  \ \ \ \xi\in \dom{A+B}.\\
\dom{AB}&=\{\xi\in \dom{B};B\xi\in \dom{A}\},\ \ \ (AB)\xi:=A(B\xi),\ \ \ \ \ \ \ \xi \in \dom{AB}.
}
In particular, $\dom{uAu^*}=u\cdot \dom{A}$ for $u\in \mathcal{U}(H)$. 
Note that even if $A,B$ are densely defined, it is very likely that $\dom{A+B}$ (or $\dom{AB}$) is not dense, whence the sum of self-adjoint operators may fail to be (essentially) self-adjoint. In fact, Israel \cite{Israel04} has shown that if $A$ has $\sigma_{\rm{ess}}(A)=\emptyset$ (see below for definition), then $\{u\in \mathcal{U}(H); \dom{A}\cap u\cdot \dom{A}=\{0\}\}$ is a norm-dense $G_{\delta}$ subset of $\mathcal{U}(H)$. 
Therefore for (norm-) generic $u$, $A+uAu^*$ is not even densely defined: $\dom{A+uAu^*}=\{0\}$.  However, if $A,B$ are self-adjoint and $B$ is bounded, then $A+B$ is always self-adjoint (with domain $\dom{A}$). 

For $A\in {\rm{SA}}(H)$, the {\it spectra} (resp. {\it point spectra}) of $A$ is denoted $\sigma(A)$ (resp. $\sigma_{\rm{p}}(A)$). The {\it essential spectra} of  $A$, denoted $\sigma_{\rm{ess}}(A)$, is the set of all $\lambda \in \sigma(A)$ which is either (i) an eigenvalue of $A$ of infinite multiplicity or (ii) an accumulation point in $\sigma(A)$. Its complement $\sigma_{\rm{d}}(A):=\sigma(A)\setminus \sigma_{\rm{ess}}(A)$ is called the {\it discrete spectra}, which is the set of all isolated eigenvalues of finite multiplicity.
The next theorem is due to Weyl (see \cite[Proposition 8.11]{Schmudgen}):
\begin{theorem}[Weyl's criterion]\label{thm: Weyl criterion} 
Let $A\in {\rm{SA}}(H)$ and $\lambda \in \mathbb{R}$. Then $\lambda\in \sigma_{\rm{ess}}(A)$ if and only if there exists a sequence $\{\xi_n\}_{n=1}^{\infty}\subset \dom{A}$ of unit vectors which converges weakly to 0, such that $\displaystyle \lim_{n\to \infty}\|A\xi_n-\lambda \xi_n\|=0$. 
\end{theorem}
The spectral measure of $A$ is denoted $E_A(\cdot)$, and we write the spectral resolution of $A$ as $A=\int_{\mathbb{R}}\lambda dE_A(\lambda)$. 
$A$ is called {\it diagonalizable} if there exists a CONS $\{\xi_n\}_{n=1}^{\infty}$ consisting of eigenvectors of $A$. Let $a_n\in \mathbb{R}$ be the eigenvalue of $A$ corresponding to $\xi_n\ (n\in \mathbb{N})$. Then $\dom{A}=\{\xi \in H; \sum_{n=1}^{\infty}a_n^2|\nai{\xi_n}{\xi}|^2<\infty\}$, and  $A\xi=\sum_{n=1}^{\infty}a_n\nai{\xi_n}{\xi}\xi_n\ (\xi \in \dom{A})$. Therefore the spectral resolution of $A$ is written as 
\[A=\sum_{n=1}^{\infty}a_n\nai{\xi_n}{\ \cdot\ }\xi_n=\sum_{n=1}^{\infty}a_ne_n,\]
where $e_n$ is the projection onto $\mathbb{C}\xi_n\ (n\in \mathbb{N})$. 
Finally, we will also need results about operator ranges:
\begin{definition}\cite{FillmoreWilliams}
We say that a subspace $\mathcal{R}\subset H$ is an {\it operator range} in $H$, if  $\mathcal{R}=\text{Ran}(T)$ for some $T\in \mathbb{B}(H)$. 
We may choose $T$ to be self-adjoint with $0\le T\le 1$. If we put $H_n:=E_T((2^{-n-1},2^{-n}])H\ (n=0,1,\cdots)$, then $H_n$ are pairwise orthogonal closed subspaces of $H$ with $H=\bigoplus_{n=0}^{\infty}H_n$ (by the density of $\mathcal{R}$). 
We call $\{H_n\}_{n=0}^{\infty}$ the {\it associated subspaces for} $T$ (see \cite[$\S$3]{FillmoreWilliams}  for details). 
\end{definition}
In this paper, we only consider dense operator ranges. The relevance of dense operator ranges to our study is the fact that a dense subspace $\mathcal{R}$ is an operator range, if and only if $\mathcal{R}=\dom{A}$ for some $A\in {\rm{SA}}(H)$ (\cite[Theorem 1.1]{FillmoreWilliams}). Every operator range (not equal to $H$) is a meager $F_{\sigma}$ subspace of $H$. 
We use the following result on the unitary equivalence of operator ranges. The essential idea of the next result, which is one of the key ingredients in the classification of operator ranges, is due to K\"othe \cite{Koethe}, and it is formulated in this form by Fillmore-Williams \cite[Theorem 3.3]{FillmoreWilliams}.  
\begin{theorem}[K\"othe, Fillmore-Williams]\label{thm: unitary equivalence of operator ranges}
Let $\mathcal{R}$, $\mathcal{S}$ be dense orator ranges in $H$ with associated closed subspaces $\{H_n\}_{n=0}^{\infty}$ and $\{K_n\}_{n=0}^{\infty}$, respectively. Then there exists $u\in \mathcal{U}(H)$ such that $u\mathcal{R}=\mathcal{S}$, if and only if the following condition is satisfied:\\
There exists $k\ge 0$ such that for each $n\ge 0$ and $l\ge 0$,
\eqa{
\dim(H_n\oplus  \cdots \oplus H_{n+l})& \le \dim (K_{n-k}\oplus \cdots \oplus K_{n+l+k})\\
\dim(K_n\oplus  \cdots \oplus K_{n+l})& \le \dim (H_{n-k}\oplus \cdots \oplus H_{n+l+k}),
}
where $H_m=K_m=\{0\}$ for $m<0$.
\end{theorem}   
We remark that the classification of dense operator ranges up to unitary equivalence is completed by Lassner-Timmermann \cite{LassnerTimmermann}. 

\subsection{Borel Equivalence Relations}\label{subsec: Borel equivalence relations}
Here we recall basic notions from (classical) descriptive set theory. The details can be found e.g., in \cite{Gao09,Hjorth00,Kechris96}. 
A topological space $X$ is called {\it Polish} if it is separable and completely metrizable. 
Let $X$ be a Polish space. We say $A\subset X$ is {\it nowhere-dense} if its closure $\overline{A}$ has empty interior. $A$ is called {\it meager} (resp. {\it comeager}), if it is (resp. its complement is) a countable union of nowhere-dense subsets of $X$. By Baire category Theorem, countable intersection of open dense subsets of $X$ is dense. 
We always assume that a Polish space is equipped with the $\sigma$-algebra $\mathcal{B}(X)$ generated by open subsets of $X$. Elements of $\mathcal{B}(X)$ are called {\it Borel} sets. 
A measurable space $(X,\mathcal{S})$ is called a {\it standard Borel space}, if it is isomorphic as a measurable space to some $(Y,\mathcal{B}(Y))$ where $Y$ is a Polish space.  A map between standard Borel spaces is called {\it Borel} if every inverse image of a Borel set is Borel. Let $X$ be a standard Borel space. If $A\subset X$ is a Borel subset, then the measurable space $(A,\mathcal{B}(X)\cap A)$ is again a standard Borel space. 
A subset $A\subset X$ is called {\it analytic} (resp. {\it co-analytic}) if it is  (resp. its complement is) the continuous image of a Polish space. An equivalence relation $E$ on $X$ is called {\it Borel}, {\it analytic} or {\it co-analytic}, respectively, if $E$ is a Borel, analytic or co-analytic subset of $X\times X$ respectively, where we  equip $X\times X$ with the product Borel structure. For an equivalence relation $E$ on $X$ and $x,y\in X$, we denote $xEy$ if $(x,y)\in E$. The set $[x]_E:=\{y\in X;\ xEy\}$ is called the $E$-{\it class} (or $E$-{\it orbit}) of $x$. Typical equivalence relations are given by group actions. Suppose  $G$ is a Polish group acting on a Polish (resp. standard Borel) space $X$. We say that the action is {\it continuous} (resp. {\it Borel}) if the action map $G\times X\ni (g,x)\mapsto g\cdot x\in X$ is continuous (resp. {\it Borel}). In this case we say that the space $X$ equipped with this action is a {\it Polish $G$-space} (resp. {\it Borel $G$-space}). The action induces an equivalence relation, called the {\it orbit equivalence relation} (denoted $E_G^X$) on $X$ given by $xE_G^Xy\ (x,y\in X)$ if and only if $y=g\cdot x$ for some $g\in G$.  In this case we write the class $[x]_{E_G^X}$ of $x\in X$ as $[x]_G$ and call it the {\it $G$-orbit} of $x\in X$. The {\it identity equivalence relation} on $X$, denoted $\text{id}_X$ is defined as $x\ \text{id}_X\ y\Leftrightarrow x=y$. 
\begin{definition}
Let $E$ (resp. $F$) be an equivalence relation on a standard Borel space $X$ (resp. $Y$). We say that $E$ is {\it Borel reducible} to $F$, in symbols $E\le_BF$, if there is a Borel map $f\colon X\to Y$ such that $x_1Ex_2\Leftrightarrow f(x_1)Ff(x_2)$ holds for every $x_1,x_2\in X$. 
\end{definition}
\begin{definition}
Let $E$ be an equivalence relation on a standard Borel space $X$. We say that $E$ is {\it smooth}, if $E$ is Borel reducible to the identity relation ${\text{id}}_Y$ on some Polish space $Y$.
\end{definition}
Finally, let $X,Y$ be Polish spaces. We say that a subset $A\subset X$ has the {\it Baire property} if there exists an open set $U\subset X$ such that the symmetric difference $A\bigtriangleup U$ is meager. 
A map $f\colon X\to Y$ is {\it Baire measurable} if $f^{-1}(V)$ has the Baire property for every open $V\subset Y$. We will use the fact that analytic sets have the Baire property (see \cite{Kechris96}).  
\subsection{Hjorth Turbulence Theorem}
The notion of classification by countable structures lies at a higher level of complexity than smoothness. 
In order to avoid introducing concepts from logic, let us informally give its definition. We refer the reader to \cite[$\S$2]{Hjorth00} for the details.  However, let us remind the reader that basic knowledge about Baire category methods and operator theory suffice to follow the arguments in later sections. 
\begin{definition}
We say that an equivalence relation $E$ admits {\it classification by countable structures}, if there exists a countable language $L$ such that $E$ is Borel reducible to the isomorphism relation on the space $X_L$ of countable $L$-structures induced by the logic action of the group $S_{\infty}$ of all permutations of $\mathbb{N}$ on $X_L$.
\end{definition}
It is known that every $S_{\infty}$-orbit equivalence relation is Borel reducible to some $E_{S_{\infty}}^{X_L}$.   Therefore $E$ admits classification by countable structures, if and only if $E\le_BE_{S_{\infty}}^X$ for some Polish $S_{\infty}$-space $X$. 
Recall also the notion of generic ergodicity of equivalence relations:
\begin{definition}
Let $E$ (resp. $F$) be an equivalence relation on a Polish space $X$ (resp. $Y$). A {\it homomorphism} from $E$ to $F$ is a map $f\colon X\to Y$ such that $xEy\Rightarrow f(x)Ff(y)$ for all $x,y\in X$. We say that $E$ is {\it generically $F$-ergodic}, if for every Baire measurable homomorphism $f$ from $E$ to $F$, there is a comeager set $A\subset X$ which $f$ maps into a single $F$-class. 
\end{definition}
Hjorth's notion of turbulence provides us with a convenient criterion for finding an obstruction of a given equivalence relation to be classifiable by countable structures. 
Below we use a category quantifier\ $\forall^*$. Suppose that we are given a Polish space $X$ and for each point $x\in X$ a proposition $P(x)$. We say that $P(x)$ holds for generic $x\in X$, denoted $\forall^*x\ (P(x))$, if $\{x\in X;P(x)\}$ is comeager in $X$.
\begin{definition}
 Let $G$ be a Polish group and $X$ a Polish $G$-space.
\begin{list}{}{}
\item[(1)]
Let $x\in X$. For an open neighborhoods $U$ of $x$ in $X$ and $V$ of $1$ in $G$, the {\it local }$U$-$V${\it orbit} of $x$, denoted $\mathcal{O}(x,U,V)$, is the set of all $y\in U$ for which there exist $l\in \mathbb{N}$, $x=x_0,x_1,\cdots,x_l=y\in U$, and $g_0,\cdots,g_{l-1}\in V$, such that $x_{i+1}=g_i\cdot x_i$ for all $0\le i\le l-1$. 
\item[(2)] The action $\alpha$ is {\it turbulent} at $x\in X$ if the local orbits $\mathcal{O}(x,U,V)$ of $x$ are somewhere dense (i.e., its closure has nonempty interior) for every open $U\subset X$ and $V\subset G$ with $x\in U$ and $1\in V$.
\item[(3)] The action $\alpha$ is said to be {\it generically turbulent} if 
\begin{list}{}{}
\item[(a)] There is a dense orbit.
\item[(b)] Every orbit is meager.
\item[(c)] $\forall^*x\in X$ [The action is turbulent at $x$]. 
\end{list}
 
\end{list}
\end{definition}
We use an apparently weaker notion of weak generic turbulence, which is actually equivalent to generic turbulence:
\begin{definition}\label{def: weak turbulence} Let $G$ be a Polish group and $X$ a Polish $G$-space. We say that the action is {\it weakly generically turbulent}, if 
\begin{itemize}
\item[(a)] Every orbit is meager.
\item[(b)] $\forall^*x\in X\ \forall^*y\in X\ \forall (\emptyset \neq)U\stackrel{\text{open}}{\subset}X\ (1\in )\forall V\stackrel{\text{open}}{\subset}G$ $[x\in U\Rightarrow \overline{\mathcal{O}(x,U,V)}\cap [y]_G\neq \emptyset].$
\end{itemize}
\end{definition}
Next theorem is called Hjorth turbulence Theorem. Proof can be found in \cite{Hjorth00}.
\begin{theorem}[Hjorth]\label{thm: weak turbulence=turbulence}
Let $G$ be a Polish group and $X$ a Polish $G$-space with every orbit meager and some orbit dense. Then the following statements are equivalent:
\begin{list}{}{}
\item[{\rm{(i)}}] $X$ is weakly generically turbulent.
\item[{\rm{(ii)}}] For any Borel $S_{\infty}$-space $Y$, $E_G^X$ is generically $E_{S_{\infty}}^Y$-ergodic.
\item[{\rm{(iii)}}] $X$ is generically turbulent.
\end{list}
\end{theorem}
It follows from Theorem \ref{thm: weak turbulence=turbulence} that orbit equivalence relations of generically turbulent actions do not admit classification by countable structures.  
\subsection{The Effros Borel Space of Closed Subsets of a Polish Space}\label{subsec: Effros Borel space}
Let $X$ be a topological space. We denote by $\mathcal{F}(X)$ the set of closed subsets of $X$. 
\begin{definition}
The {\it Effros Borel structure} of $\mathcal{F}(X)$ is the $\sigma$-algebra on $\mathcal{F}(X)$ generated by the sets
\[\{F\in \mathcal{F}(X);\ F\cap U\neq \emptyset.\}\]
If $X$ is second countable and $\{U_n\}_{n=1}^{\infty}$ is an open basis for $X$, it is sufficient to consider $U$ in this basis.
\end{definition}
If $X$ is Polish, then the Effros Borel space of $\mathcal{F}(X)$ is a standard Borel space \cite[Theorem 12.6]{Kechris96}. 
Next result is important for our analysis (see \cite{Christensen71,Christensen74} for more details):
\begin{theorem}[Christensen]\label{thm: Christensen}
Let $X$ be a Polish space. Then the intersection map $I_2\colon \mathcal{F}(X)\times \mathcal{F}(X)\ni (K_1,K_2)\mapsto K_1\cap K_2\in \mathcal{F}(X)$ is Borel if and only if $X$ is $\sigma$-compact.
\end{theorem}
\section{Polish Space ${\rm{SA}}(H)$}\label{sec: SA(H) is Polish}
In this section, we show that the space ${\rm{SA}}(H)$ of all self-adjoint operators on $H$ equipped with strong resolvent topology (SRT) is a Polish space when $H$ is separable infinite-dimensional. This is probably known, but since we could not find a reference, we add a proof for completeness. 
Fix a CONS $\{\xi_n\}_{n=1}^{\infty}$ for $H$, and define a metric $d$ on ${\rm{SA}}(H)$ by 
\[d(A,B):=\sum_{n=1}^{\infty}\sum_{m=1}^{\infty}\frac{1}{2^{n+m}}\sup_{t\in [-m,m]}\|e^{itA}\xi_n-e^{itB}\xi_n\|,\ \ \ \ A,B\in {\rm{SA}}(H).\]
\begin{proposition}\label{prop: SA(H) is Polish}
$d$ is a complete metric on ${\rm{SA}}(H)$ compatible with SRT, and ${\rm{SA}}(H)$ is separable with respect to SRT. Consequently, ${\rm{SA}}(H)$ is a Polish space.
\end{proposition}
\begin{remark}
Note that the following apparently suitable compatible metric $d'$ given by 
\[d'(A,B):=\sum_{n=1}^{\infty}\frac{1}{2^n}\|(A-i)^{-1}\xi_n-(B-i)^{-1}\xi_n\|\]
is not complete. Indeed, $A_n=n1\in {\rm{SA}}(H)\ (n\in \mathbb{N})$ is $d'$-Cauchy which has no SRT-limit. 
\end{remark}
We need the following classical result. Proof can be found e.g. in \cite[Theorem VIII.21]{ReedSimonI}. 
\begin{lemma}[Trotter]\label{lem: Trotter lemma}
Let $\{A_k\}_{n=1}^{\infty}\subset {\rm{SA}}(H)$. Then $A_k$ converges to $A\in {\rm{SA}}(H)$ in SRT, if and only if for each $\xi\in H$ and each compact subset $K$ of $\mathbb{R}$, $\sup_{t\in K}\|e^{itA_k}\xi-e^{itA}\xi\|$ tends to 0.
\end{lemma}
\begin{proof}[Proof of Proposition \ref{prop: SA(H) is Polish}]
We first show that ${\rm{SA}}(H)$ is separable. Let $F:{\rm{SA}}(H)\to \prod_{n\in \mathbb{N}}H$ be a map defined by $F(A):=((A-i)^{-1}\xi_n)_{n=1}^{\infty}$. 
Then $F$ is injective. Indeed, if $F(A_1)=F(A_2)$ for $A_1,A_2\in {\rm{SA}}(H)$, then as resolvents are bounded, $(A_1-i)^{-1}=(A_2-i)^{-1}$, which implies $A_1=A_2$. 
It is also easy to see that for a net $\{A_{\alpha}\}$ and $A$ in ${\rm{SA}}(H)\ (k\in \mathbb{N})$,
\eqa{
A_{\alpha}\stackrel{{\rm{SRT}}}{\to} A&\Leftrightarrow (A_{\alpha}-i)^{-1}\xi \to (A-i)^{-1}\xi,\ \ \xi\in H\\
&\Leftrightarrow (A_{\alpha}-i)^{-1}\xi_n \to (A-i)^{-1}\xi_n,\ \ n\in \mathbb{N}.
}
Hence $F$ is a homeomorphism of ${\rm{SA}}(H)$ onto its range. Therefore as $\prod_{n\in \mathbb{N}}H$ is Polish, its subspace $F({\rm{SA}}(H))$ is separable and metrizable, whence so is ${\rm{SA}}(H)$. 
Next, we show that $d$ is a metric compatible with SRT.\\
It is easy to see that $d$ is a metric, and note that by Lemma \ref{lem: Trotter lemma} we have the following equivalence (we set $I_m=[-m,m]$):
\eqa{
d(A_k,A)\stackrel{k\to \infty}{\to}0&\Leftrightarrow \sup_{t\in I_m}\|(e^{itA_k}-e^{itA})\xi_n\|\stackrel{k\to \infty}{\to}0,\ \ \ \ n,m\in \mathbb{N}\\
&\Leftrightarrow \sup_{t\in I_m}\|(e^{itA_k}-e^{itA})\xi\|\stackrel{k\to \infty}{\to}0,\ \ \ \ \xi\in H, m\in \mathbb{N}\\
&\Leftrightarrow A_k\stackrel{\rm{SRT}}{\to}A.
}
Therefore $d$ is compatible with SRT. Finally, we show that $d$ is complete. Suppose that $\{A_k\}_{k=1}^{\infty}$ is a $d$-Cauchy sequence in ${\rm{SA}}(H)$. Then for each $n,m\in \mathbb{N}$, we have 
\begin{equation}
\sup_{t\in I_m}\|(e^{itA_k}-e^{itA_l})\xi_n\|\stackrel{k,l\to \infty}{\to}0.\label{eq: d-Cauchy e_n}
\end{equation}
Now fix $t\in \mathbb{R}$ and let $\xi\in H$. We show that $\{e^{itA_k}\xi\}_{k=1}^{\infty}$ is Cauchy in $H$. Given $\varepsilon>0$, find $\xi_0\in \text{span}\{\xi_n;n\ge 1\}$ such that $\|\xi-\xi_0\|<\varepsilon/4$. By (\ref{eq: d-Cauchy e_n}), we see that $\{e^{itA_k}\xi_0\}_{k=1}^{\infty}$ is Cauchy in $H$. Therefore there exists $k_0$ such that $\|e^{itA_k}\xi_0-e^{itA_l}\xi_0\|<\varepsilon/2$ for all $k,l\ge k_0$. Then for $k,l\ge k_0$, 
\eqa{
\|e^{itA_k}\xi-e^{itA_l}\xi\|&\le \|(e^{itA_k}-e^{itA_l})(\xi-\xi_0)\|+\|(e^{itA_k}-e^{itA_l})\xi_0\|\\
&\le 2\|\xi-\xi_0\|+\varepsilon/2<\varepsilon.
}
Therefore $\{e^{itA_k}\xi\}_{k=1}^{\infty}$ is Cauchy, and let $\varphi(t,\xi)\in H$ be its limit. It is easy to see that for a fixed $t\in \mathbb{R}$, $\xi\mapsto \varphi(t,\xi)$ is linear. Moreover, $\|\varphi(t,\xi)\|=\|\xi\|$ for each $t\in \mathbb{R}$, $\xi\in H$. Therefore for each $t\in \mathbb{R}$, there exists an isometry $u(t)\in \mathbb{B}(H)$ such that $\varphi(t,\xi)=u(t)\xi\ (\xi\in H)$. It is clear that $u(0)=1$. Moreover, for $s,t\in \mathbb{R}$ and $\xi\in H$, it holds that
\eqa{
\| u(s)u(t)\xi-u(s+t)\xi \| &=\lim_{k\to \infty}\| e^{isA_k}u(t)\xi-e^{i(s+t)A_k}\xi \|\\
&=\lim_{k\to \infty}\|u(t)\xi-e^{itA_k}\xi\|\\
&=\| u(t)\xi-u(t)\xi \|=0,
}
which implies that $\{u(t)\}_{t\in \mathbb{R}}$ is a one-parameter unitary group. 
We show that $t\mapsto u(t)$ is strongly continuous. 
Since it is a one-parameter unitary group (hence uniformly bounded), it suffices to show that $t\mapsto u(t)\xi_n$ is continuous at $t=0$ for each $n\in \mathbb{N}$. So let $\varepsilon>0$ and $n\in \mathbb{N}$ be given. By (\ref{eq: d-Cauchy e_n}), there exists $k_0\in \mathbb{N}$ such that for each $t\in [-1,1]$ and $k,l\ge k_0$,
\begin{equation}
\|(e^{itA_k}-e^{itA_l})\xi_n\|<\frac{\varepsilon}{2}.\label{eq: Cauchy e_n}
\end{equation}
Letting $l\to \infty$ in (\ref{eq: Cauchy e_n}), we obtain that $\|(e^{itA_k}-u(t))\xi_n\|\le \frac{\varepsilon}{2}$ for each $t\in [-1,1]$ and $k\ge k_0$. 
On the other hand, there exists $(1>)\delta>0$ such that 
$\|e^{itA_{k_0}}\xi_n-\xi_n\|<\frac{\varepsilon}{2}$ for $|t|<\delta$. Therefore for $|t|<\delta$, we obtain
\eqa{
\|u(t)\xi_n-\xi_n\|&\le \|u(t)\xi_n-e^{itA_{k_0}}\xi_n\|+\|e^{itA_{k_0}}\xi_n-\xi_n\|\\
&<\varepsilon.
}
Therefore $t\mapsto u(t)$ is strongly continuous, and by Stone Theorem \cite[Theorem VIII.8]{ReedSimonI} let $A\in {\rm{SA}}(H)$ be such that $u(t)=e^{itA}\ (t\in \mathbb{R})$ holds. We show that $d(A_k,A)\stackrel{k\to \infty}{\to}0$. To this purpose it suffices to show that $\sup_{t\in I_m}\|(e^{itA_k}-e^{itA})\xi_n\|$ tends to 0 for each $n,m\in \mathbb{N}$. Let $\varepsilon>0$. Then there exists $k_0\in \mathbb{N}$ such that for all $k,l\ge k_0$, we have 
\[\sup_{t\in I_m}\|(e^{itA_k}-e^{itA_l})\xi_n\|<\frac{\varepsilon}{2}.\] 
Now for each $t\in I_m$, choose $l=l(t)\ge k_0$ such that $\|(e^{itA}-e^{itA_l})\xi_n\|<\varepsilon/2$. It then follows that for all $k\ge k_0$, 
\eqa{
\sup_{t\in I_m}\|(e^{itA}-e^{itA_k})\xi_n\|&\le \sup_{t\in I_m}\|(e^{itA}-e^{itA_{l(t)}})\xi_n\|+
\sup_{t\in I_m}\|(e^{itA_{l(t)}}-e^{itA_k})\xi_n\|<\varepsilon,
}
whence the claim is proved. Therefore $d$ is complete. 
 \end{proof}
\section{Weyl-von Neumann Equivalence Relation $E_G^{{\rm{SA}}(H)}$}
\subsection{Impossibility of von Neumann's Theorem for Unbounded Self-Adjoint Operators}
Let $H$ be a separable infinite-dimensional Hilbert space. 
von Neumann's Theorem ((1)$\Rightarrow $(2) of Theorem \ref{thm: von Neumann theorem}) asserts that bounded self-adjoint operators $A,B\in \mathbb{B}(H)_{\rm{sa}}$ with the same essential spectra  $\sigma_{\rm{ess}}(A)=\sigma_{\rm{ess}}(B)$ are unitarily equivalent modulo compacts, i.e., $B=uAu^*+K$ for some $u\in \mathcal{U}(H)$ and $K\in \mathbb{K}(H)$. 
In this section we consider the situation for unbounded self-adjoint operators (note that Weyl's Theorem (2)$\Rightarrow $(1) of Theorem \ref{thm: von Neumann theorem} holds in full generality):
\begin{question}\label{quest: weyl-von Neumann equivalence}
Let $A,B\in {\rm{SA}}(H)$ be such that $\sigma_{\rm{ess}}(A)=\sigma_{\rm{ess}}(B)$. 
Are there $u\in \mathcal{U}(H)$ and $K\in \mathbb{K}(H)_{\rm{sa}}$ such that $uAu^*+K=B$?
\end{question}
The answer to the question is negative, as the following simple example shows:
\begin{example}\label{ex: ess spec=0 counterexample}
Let $H_0$ be a separable infinite-dimensional Hilbert space, and let $H=H_0\oplus H_0$. 
 Fix a CONS $\{\xi_n\}_{n=1}^{\infty}$  for $H_0$, and let $A_0:=\sum_{n=1}^{\infty}n\nai{\xi_n}{\ \cdot\ }\xi_n\in {\rm{SA}}(H_0)$, and define 
$A,B\in {\rm{SA}}(H)$ by
\[A:=A_0\oplus 0,\ \ B:=0\oplus 0.\]
Then $\sigma_{\rm{ess}}(A)=\sigma_{\rm{ess}}(B)=\{0\}$, and since $A$ is unbounded, so is $uAu^*+K$ for any $u\in \mathcal{U}(H)$ and $K\in {\rm{SA}}(H)$.  Thus $uAu^*+K\neq B$. 
\end{example}
It is now clear why von Neumann's Theorem fails to hold for unbounded self-adjoint operators: if $A,B$ are unitarily equivalent modulo compacts, then their domains $\dom{A}$ and $\dom{B}$ must be unitarily equivalent, i.e., $u\cdot \dom{A}=\dom{B}$ for some $u\in \mathcal{U}(H)$. In fact there are a lot of unbounded self-adjoint operators with the same essential spectra but have non-unitarily equivalent domains. We give one such example: 
\begin{example}\label{ex: non-unitarily equivalent but have same ess spec}
Let $\{\xi_n\}_{n=0}^{\infty}$ be a fixed CONS for $H$. Let $e_n$ be the projection of $H$ onto $\mathbb{C}\xi_n$. 
Define $\{A_t\}_{t\in (0,1)}\subset {\rm{SA}}(H)$ by 
\[A_t=\sum_{n=1}^{\infty}2^{(n^t)}e_n\ \ \ \ (0<t<1).\]
We show that $\{A_t\}_{t\in (0,1)}$ is a family of self-adjoint operators with $\sigma_{\rm{ess}}(A_t)=\emptyset\ (0<t<1)$ such that $\dom{A_t}$ and $\dom{A_s}$ are not unitarily equivalent for $0<t\neq s<1$. The first assertion is clear, since $2^{n^t}\stackrel{n\to \infty}{\to}\infty$. Fro each $0<t<1$, the domain of $A_t$ is $\dom{A_t}=\text{Ran}(A_t^{-1})$, where $A_t^{-1}=\sum_{n=1}^{\infty}2^{-n^t}e_n$. Therefore the associated subspaces for $A_t^{-1}$ are
\eqa{
H_n^{(t)}&:=E_{A_t^{-1}}((2^{-n-1},2^{-n}])H\\
&=\overline{\text{span}}\{\xi_k;2^{-n-1}<2^{-k^t}\le 2^{-n}\}\\
&=\overline{\text{span}}\{\xi_k; n^{\frac{1}{t}}\le k<(n+1)^{\frac{1}{t}}\}.
}
Let $0<t<s<1$. Then 
\eqa{
\dim(H_n^{(t)})&\ge [(n+1)^{\frac{1}{t}}]-([n^{\frac{1}{t}}]+1)\\
&\ge (n+1)^{\frac{1}{t}}-n^{\frac{1}{t}}-2,\\
\dim(H_n^{(s)})&\le (n+1)^{\frac{1}{s}}-n^{\frac{1}{s}}.
}
Therefore for given $k,l\in \mathbb{N}$, and $n>k$, it holds that 
\eqa{
\dim(H_{n-k}^{(s)}\oplus \cdots H_{n+l+k}^{(s)})&\le \sum_{m=n-k}^{n+l+k}\{(m+1)^{\frac{1}{s}}-m^{\frac{1}{s}}\}\\
&=(n+l+k)^{\frac{1}{s}}-(n-k)^{\frac{1}{s}}\\
&\stackrel{(*)}{\le} (l+2k)s^{-1}(n+l+k)^{s^{-1}-1},\\
\dim(H_n^{(t)}\oplus \cdots \oplus H_{n+l}^{(t)})&\ge \sum_{m=n}^{n+l}\{(m+1)^{\frac{1}{t}}-m^{\frac{1}{t}}-2\}\\
&=(n+l)^{\frac{1}{t}}-n^{\frac{1}{t}}-2(l+1)\\
&\stackrel{(*)}{\ge} lt^{-1}n^{t^{-1}-1}-2(l+1),
}
where we used the mean value Theorem in $(*)$. 
Since $t^{-1}>s^{-1}>1$, it holds that
\[\lim_{n\to \infty}\frac{lt^{-1}n^{t^{-1}-1}-2(l+1)}{(l+2k)s^{-1}(n+l+k)^{s^{-1}-1}}=\infty,\]
which in particular shows that $\dim(H_n^{(t)}\oplus \cdots \oplus H_{n+l}^{(t)})>\dim(H_{n-k}^{(s)}\oplus \cdots H_{n+l+k}^{(s)})$ for large $n$. Since $k,l$ are arbitrary, $\dom{A_t}$ and $\dom{A_s}$ are not unitarily equivalent by Theorem \ref{thm: unitary equivalence of operator ranges}. 
\end{example} 
We next show that unitary equivalence of the domains is still insufficient. Namely we construct 
another continuous family $\{B_t\}_{t\in [0,1]}$ in ${\rm{SA}}(H)$ with the same domain and the essential spectra, yet no two of them are unitarily equivalent modulo compacts. 

\begin{example}\label{ex: unitarily quivalent domains but not uAu*+K=B}
Let $\{\xi_n\}_{n=1}^{\infty}$ and $\{e_n\}_{n=1}^{\infty}$ be as in Example \ref{ex: non-unitarily equivalent but have same ess spec}. Fix a bijection $\nai{\cdot}{\cdot}\colon \mathbb{N}^2\to \mathbb{N}$ given by 
\[\nai{k}{m}:=2^{k-1}(2m-1),\ \ \ \ \ m,k\in \mathbb{N}.\]
and define a family $\{B_t\}_{t\in [0,1]}\subset {\rm{SA}}(H)$ by
\[B_t:=\sum_{n=1}^{\infty}\lambda_n^{(t)}e_n,\ \ \ \lambda_{\nai{k}{m}}^{(t)}:=k+\frac{t}{m+2},\ \ \ \ \ \ t\in [0,1],\ k,m\in \mathbb{N}.\]
It is easy to see that $\dom{B_s}=\dom{B_t}$, and $\sigma_{\rm{ess}}(B_s)=\sigma_{\rm{ess}}(B_t)=\mathbb{N}\ (s,t\in [0,1])$.  
Let $0\le s<t\le 1$. We then show that there are no $u\in \mathcal{U}(H)$ and $K\in \mathbb{K}(H)_{\rm{sa}}$ satisfying $uB_tu^*+K=B_s$. 
Suppose by contradiction that there exist such $u$ and $K$, and put  $\eta_n:=u\xi_n\ (n\in \mathbb{N})$. Then $f_n:=ue_nu^*$ is a projection onto $\mathbb{C}\eta_n$, and
\[\sum_{k,m=1}^{\infty}\left (k+\frac{t}{m+2}\right )f_{\nai{k}{m}}+K=\sum_{k,m=1}^{\infty}\left (k+\frac{s}{m+2}\right )e_{\nai{k}{m}}.\]
Apply the above equality to the vector $\eta_{\nai{k}{1}}\ (k\in \mathbb{N})$ to obtain
\[\left (k+\frac{t}{3}\right )\eta_{\nai{k}{1}}+K\eta_{\nai{k}{1}}=\sum_{l,m=1}^{\infty}\left (l+\frac{s}{m+2}\right )e_{\nai{l}{m}}\eta_{\nai{k}{1}}.\]
Since $K$ is compact and $\nai{k}{1}=2^{k-1}$, $\eta_{\nai{k}{1}}\stackrel{k\to \infty}{\to} 0$ weakly, we have $\|K\eta_{\nai{k}{1}}\|\stackrel{k\to \infty}{\to}0$. 
For $k,l,m\in \mathbb{N}$, let $a(k,l,m):=|k-l+\frac{t}{3}-\frac{s}{m+2}|$. Then $a(k,l,m)\ge |k-l|-|\frac{t}{3}-\frac{s}{m+2}|\ge 1-\frac{2}{3}=\frac{1}{3}$ if $k\neq l$, while $a(k,k,m)=\frac{t}{3}-\frac{s}{m+2}\ge \frac{t-s}{3}$. Therefore 
\[\textstyle a(k,l,m)\ge \delta(t,s):=\frac{1}{3}(t-s)>0,\ \ \ (k,l,m\in \mathbb{N}).\] 
From this we have
\eqa{
\|K\eta_{\nai{k}{1}}\|^2&=\left \|\left (k+\frac{t}{3}\right )\eta_{\nai{k}{1}}-\sum_{l,m=1}^{\infty}\left (l+\frac{s}{m+2}\right )e_{\nai{l}{m}}\eta_{\nai{k}{1}}\right \|^2\\
&=\sum_{l,m=1}^{\infty}\left |k-l+\frac{t}{3}-\frac{s}{m+2}\right |^2\|e_{\nai{l}{m}}\eta_{\nai{k}{1}}\|^2\\
&\ge \delta(t,s)^2\sum_{l,m=1}^{\infty}\|e_{\nai{l}{m}}\eta_{\nai{k}{1}}\|^2=\delta(t,s)^2.
}
This is a contradiction to $\lim_{k\to \infty}\|K\eta_{\nai{k}{1}}\|=0$. 
\end{example} 
Taking all the above examples into account, it seems unlikely that there exists a complete invariant for the von Neumann type classification problem for ${\rm{SA}}(H)$, such that the assignment of the invariant is constructible in some sense.  
\subsection{Orbit Equivalence Relation $E_G^{{\rm{SA}}(H)}$}
To consider the complexity of the classification problem of self-adjoint operators up to unitary equivalence modulo compact perturbations we use ${\rm{SA}}(H)$ as parameter Polish space (see $\S$\ref{sec: SA(H) is Polish}) and regard the equivalence as orbit equivalence of a Polish group: 
\begin{definition}
(1)\ We define the Polish group $G$ to be the semidirect product $\mathbb{K}(H)_{\rm{sa}}\rtimes \mathcal{U}(H)$, where $\mathbb{K}(H)_{\rm{sa}}$ is the additive Polish group of compact self-adjoint operators with the norm topology, and we equip $\mathcal{U}(H)$ with SOT. The action of $\mathcal{U}(H)$ on $\mathbb{K}(H)_{\rm{sa}}$ is given by conjugation: 
\[u\cdot K:=uKu^*,\ \ \ \ \ \ u\in \mathcal{U}(H),\ K\in \mathbb{K}(H)_{\rm{sa}}.\]
Thus $(0,1)$ is the identity of $G$ and the group law on $G$ is given by 
\[(K_1,u_1)\cdot (K_2,u_2)=(K_1+u_1K_2u_1^*,u_1u_2),\ \ \ \ u_i\in \mathcal{U}(H),\ K_i\in \mathbb{K}(H)_{\rm{sa}}\ (i=1,2).\]
(2) We define the action $\alpha\colon G\times {\rm{SA}}(H)\to {\rm{SA}}(H)$ by
\[(K,u)\cdot A:=uAu^*+K,\ \ \ \ \,\ A\in {\rm{SA}}(H),\ u\in \mathcal{U}(H),\ K\in \mathbb{K}(H)_{\rm{sa}}.\]
\end{definition} 
It is easy to see that $\alpha$ is indeed an action. Therefore the classification problem in consideration is nothing but the study of the Borel complexity of the orbit equivalence relation $E_G^{{\rm{SA}}(H)}$.  
\begin{definition}
We call $E_G^{{\rm{SA}}(H)}$ the {\it Weyl-von Neumann equivalence relation}. 
\end{definition}
Next we show that ${\rm{SA}}(H)$ is a Polish $G$-space. 
\begin{proposition}\label{prop: G acts on SA(H) continuously}
The action $\alpha\colon G\curvearrowright {\rm{SA}}(H)$ is continuous. 
\end{proposition}
We first show the continuity of the $\mathbb{B}(H)_{\rm{sa}}$-action, where we equip the additive group $\mathbb{B}(H)_{\rm{sa}}$ with the norm topology.  
\begin{proposition}\label{prop: B(H)_sa acts on SA(H) continuously}
The action $\alpha_0\colon \mathbb{B}(H)_{\rm{sa}}\curvearrowright {\rm{SA}}(H)$ given by $(K,A)\mapsto A+K$ is continuous. 
\end{proposition}
The key point in the proof of Proposition \ref{prop: B(H)_sa acts on SA(H) continuously} is the next lemma, which was communicated to us by Asao Arai. We are grateful to him for allowing us to include his proof.
\begin{lemma}[Arai]\label{lem: Arai's lemma}
Let $K\in \mathbb{B}(H)_{\rm{sa}}$ and let $A_n,A\in {\rm{SA}}(H)\ (n\in \mathbb{N})$ be such that $A_n\stackrel{\rm{SRT}}{\to}A$. Then $A_n+K\stackrel{\rm{SRT}}{\to}A+K$ holds. 
\end{lemma}
For the proof, we use the following well-known result. 
\begin{lemma}{\rm{\cite[Theorem VIII.19]{ReedSimonI}}}\label{lem: SRT-convergence at two points}
Let $T_n, T\in {\rm{SA}}(H)\ (n\in \mathbb{N})$. If there exists $z\in \mathbb{C}\setminus \mathbb{R}$ such that ${\rm{SOT-}}\lim_{n\to \infty}(T_n-z)^{-1}=(T-z)^{-1}$, then ${\rm{SRT-}}\lim_{n\to \infty}T_n
=T$.
\end{lemma}
\begin{proof}[Proof of Lemma \ref{lem: Arai's lemma}]
For any $z\in \mathbb{C}\setminus \mathbb{R}$, we have
\[\|K(A_n-z)^{-1}\|\le \frac{\|K\|}{|{\rm{Im}}\ z|},\ \ \ \ \|K(A-z)^{-1}\|\le \frac{\|K\|}{|{\rm{Im}}\ z|}.\]
Therefore it holds that if $\|K\|<|{\rm{Im}}\ z|$, 
\eqa{
(A_n+K-z)^{-1}&=\sum_{k=0}^{\infty}(-1)^k(A_n-z)^{-1}(K(A_n-z)^{-1})^k,\\
(A+K-z)^{-1}&=\sum_{k=0}^{\infty}(-1)^k(A-z)^{-1}(K(A-z)^{-1})^k.
}
Therefore for arbitrary $\xi\in H$, we have
\begin{align}
&\|(A_n+K-z)^{-1}\xi-(A+K-z)^{-1}\xi\|\notag \\
&=\sum_{k=0}^{\infty}\|\{(A_n-z)^{-1}(K(A_n-z)^{-1})^k-(A-z)^{-1}(K(A-z)^{-1})^k\}\xi\|.\label{eq: Arai Neumann series}
\end{align}
Since $A_n\stackrel{n\to \infty}{\to}A$ (SRT), $K(A_n-z)^{-1}\stackrel{n\to \infty}{\to}K(A-z)^{-1}$ (SOT) holds. This implies that for each $k\ge 0$, $(A_n-z)^{-1}(K(A_n-z)^{-1})^k\stackrel{n\to \infty}{\to}(A-z)^{-1}(K(A-z)^{-1})^k$ (SOT). Therefore each term in (\ref{eq: Arai Neumann series}) tends to 0 as $n\to \infty$. Furthermore, we see that 
\begin{equation}
\|\{(A_n-z)^{-1}(K(A_n-z)^{-1})^k-(A-z)^{-1}(K(A-z)^{-1})^k\}\xi\|\le 2|{\rm{Im}}\ z|^{-1}\left (\frac{\|K\|}{|{\rm{Im}}\ z|}\right )^k\|\xi\|,\label{eq: Arai estimate 1 with xi}
\end{equation}
and since $\sum_{k=0}^{\infty}(\|K\|/|{\rm{Im}}\ z|)^k<\infty$, we have for $\|K\|<|{\rm{Im}}\ z|$ that
\begin{equation}
\lim_{n\to \infty}(A_n+K-z)^{-1}\xi=(A+K-z)^{-1}\xi.\label{eq: Arai estimate 2 with xi}
\end{equation}
Therefore by Lemma \ref{lem: SRT-convergence at two points}, $A_n+K\stackrel{n\to \infty}{\to}A+K$ (SRT) holds.
\end{proof}
\begin{proof}[Proof of Proposition \ref{prop: B(H)_sa acts on SA(H) continuously}]
Let $\{A_n\}_{n=1}^{\infty}$ (resp. $\{K_n\}_{n=1}^{\infty}$) be a sequence in ${\rm{SA}}(H)$ (resp. in $\mathbb{B}(H)_{\rm{sa}}$) converging to $A\in {\rm{SA}}(H)$ (resp. to $K\in \mathbb{B}(H)_{\rm{sa}}$). 
For any $\xi \in H$, we have 
\begin{align}
\|(A_n+&K_n-i)^{-1}\xi-(A+K-i)^{-1}\xi\|\notag \\
&\le \|\{(A_n+K_n-i)^{-1}-(A_n+K-i)^{-1}\}\xi\|+\|
\{(A_n+K-i)^{-1}-(A+K-i)^{-1}\}\xi\|\label{eq: first term easy second is Arai}.
\end{align}
By the resolvent identity \cite[$\S$2.2, (2.4)]{Schmudgen}, the first term in (\ref{eq: first term easy second is Arai}) is estimated as 
\eqa{
\|\{(A_n+K_n-i)^{-1}-(A_n+K-i)^{-1}\}\xi\|&\le \|(A_n+K_n-i)^{-1}(K_n-K)(A_n+K-i)^{-1}\xi\|\\
 &\le \|K_n-K\|\cdot \|\xi\|\stackrel{n\to \infty}{\to}0.
}
The second term in (\ref{eq: first term easy second is Arai}) also tends to 0 by Lemma \ref{lem: Arai's lemma}. Therefore $A_n+K_n\stackrel{n\to \infty}{\to}A+K$ (SRT) holds.  
\end{proof}
\begin{proof}[Proof of Proposition \ref{prop: G acts on SA(H) continuously}]
Assume that $A_n\in {\rm{SA}}(H)$ (resp. $(K_n,u_n)\in G$) converges to $A\in {\rm{SA}}(H)$ (resp. $(K,u)\in G$). Then $u_nA_nu_n^*\stackrel{n\to \infty}{\to}uAu^*$ (SRT), because the joint SOT-continuity of operator product on bounded sets shows that 
\[(u_nA_nu_n^*-i)=u_n(A_n-i)^{-1}u_n^*\stackrel{n\to \infty}{\to}u(A-i)^{-1}u^*=(uAu^*-i)^{-1}\ ({\rm{SOT}}).\]
Therefore by the continuity of $\alpha_0$ (Proposition \ref{prop: B(H)_sa acts on SA(H) continuously}), we have 
\[(K_n,u_n)\cdot A_n=u_nA_nu_n^*+K_n\stackrel{n\to \infty}{\longrightarrow } uAu^*+K=(K,u)\cdot A\ \ ({\rm{SRT}}). \]
\end{proof}

\subsection{Smoothness: Bounded Case}\label{subsec: Smoothness: Bounded Case}
Recall from $\S$\ref{subsec: Effros Borel space} that the Effros Borel structure on the space $\mathcal{F}(\mathbb{R})$ of all closed subsets of $\mathbb{R}$ is the $\sigma$-algebra generated by sets of the form $\{F\in \mathcal{F}(\mathbb{R});\ F\cap U\neq \emptyset\}$, where $U$ is an open subset of $\mathbb{R}$.  In this section, we show that Weyl-von Neumann equivalence relation restricted on $\mathbb{B}(H)_{\rm{sa}}$ is smooth by showing that ${\rm{SA}}(H)\ni A\mapsto \sigma_{\rm{ess}}(A)\in \mathcal{F}(\mathbb{R})$ is Borel. 
\begin{lemma}\label{lem: B(H)_sa is standard}
$\mathbb{B}(H)_{\rm{sa}}$ is a meager $F_{\sigma}$ subset of ${\rm{SA}}(H)$. In particular, it is Borel.  
\end{lemma}
\begin{proof}
Let $F_n:=\{A\in \mathbb{B}(H)_{\rm{sa}};\ \|A\|\le n\}\ (n\in \mathbb{N})$. Then $\mathbb{B}(H)_{\rm{sa}}=\bigcup_{n=1}^{\infty}F_n$.  We show that each $F_n$ is SRT-closed. 
Let $A_k\in F_n$ and assume that $A_k\stackrel{k\to \infty}{\to}A\in {\rm{SA}}(H)$. We show that $A\in F_n$: let $\xi \in H$, and let $A_k=\int_{-n}^n\lambda dE_{A_k}(\lambda)$ be the spectral resolution of $A_k$($k\in \mathbb{N}$). Then for each $k\ge 1$ we have
\begin{align}
\|(A_k-i)^{-1}\xi\|^2&=\int_{[-n,n]}\frac{1}{|\lambda-i|^2}d\|E_{A_k}(\lambda)\xi\|^2\ge \frac{1}{n^2+1}\|\xi\|^2.\label{eq: bounded operator and resolvent}
\end{align}
Therefore 
\begin{equation}
\|(A-i)^{-1}\xi\|^2=\lim_{k\to \infty}\|(A_k-i)^{-1}\xi\|^2\ge \frac{1}{n^2+1}\|\xi\|^2,\ \ \ \ \xi \in H.\label{eq: A is bounded}
\end{equation} 
If there exists $\lambda \in \sigma(A)\cap \mathbb{R}\setminus [-n,n]$, choose $\varepsilon>0$ such that $|\lambda|-\varepsilon>n$, and $\xi \in \dom{A}$ such that $\|A\xi-\lambda \xi\|<\varepsilon \|\xi\|$. Then $\|A\xi\|\ge \|\lambda \xi \|-\|A\xi-\lambda \xi\|>(|\lambda| -\varepsilon )\|\xi\|$, so that 
\eqa{
\|(A-i)\xi\|^2&=\nai{A\xi-i\xi}{A\xi-i\xi}=\|A\xi\|^2+\|\xi\|^2\\
&>\{(|\lambda|-\varepsilon)^2+1\}\|\xi\|^2,
}
which by (\ref{eq: A is bounded}) implies that 
\eqa{
\|\xi\|^2&=\|(A-i)^{-1}(A-i)\xi\|^2\ge \frac{1}{n^2+1}\{(|\lambda|-\varepsilon)^2+1\}\|\xi\|^2>\|\xi\|^2,
}
a contradiction. Therefore $\sigma(A)\subset [-n,n]$, and $A\in F_n$. Therefore $\mathbb{B}(H)_{\rm{sa}}$ is $F_{\sigma}$ in SA$(H)$.

Finally, we show that $F_n$ has empty interior in ${\rm{SA}}(H)$, whence $\mathbb{B}(H)_{\rm{sa}}$ is meager. Assume by contradiction that there is $A_0\in {\rm{Int}}(F_n)$. 
Then by Weyl-von Neumann Theorem \ref{thm: Weyl-von Neumann}, there exists $K_0\in \mathbb{K}(H)_{\rm{sa}}$ such that $A_0+K_0$ is in ${\rm{Int}}(F_n)$ and has the form $\sum_{m=1}^{\infty}\lambda_me_m$, where $\{e_m\}_{m=1}^{\infty}$ is a sequence of mutually orthogonal rank one projections with sum equal to 1, and $\{\lambda_m\}_{m=1}^{\infty}\subset \mathbb{R}$. 
Let $A_k:=\sum_{m=1}^{k}\lambda_me_m+\sum_{m=k+1}^{\infty}me_m$. 
Then for each $\xi\in H$, we have
\begin{align}
\|(A_k-i)^{-1}\xi-(A_0+K_0-i)^{-1}\xi\|^2&=\sum_{m=k+1}^{\infty}\left |\frac{1}{m-i}-\frac{1}{\lambda_m-i}\right |^2\|e_n\xi\|^2\notag \\
&\le \sum_{m=k+1}^{\infty}4\|e_m\xi\|^2\stackrel{k\to \infty}{\to}0.\label{eq: SRT approximation by unbounded operators}
\end{align}
Since $A_0+K_0\in \text{Int}(F_n)$, this shows that $A_k\in F_n$ for large enough $k$, which is a contradiction because each $A_k$ is unbounded.  
\end{proof}
By Lemma \ref{lem: B(H)_sa is standard}, $\mathbb{B}(H)_{\rm{sa}}$ is a standard Borel space with respect to the subspace Borel structure. Since $\mathbb{B}(H)_{\rm{sa}}$ is $G$-invariant, we may consider the restricted action of $G$ on $\mathbb{B}(H)_{\rm{sa}}$ and its orbit equivalence relation $E_G^{\mathbb{B}(H)_{\rm{sa}}}$. 
\begin{theorem}\label{thm: E_G^B(H)_sa is smooth}
$E_G^{\mathbb{B}(H)_{\rm{sa}}}$ is a smooth equivalence relation. 
\end{theorem}
\begin{lemma}\label{lem: spectrum is Borel}
The map ${\rm{SA}}(H)\ni A\mapsto \sigma(A)\in \mathcal{F}(\mathbb{R})$ is Borel.
\end{lemma}
\begin{proof}
It clearly suffices to show that for any $a,b\in \mathbb{R}\ (a<b)$, the set $\mathcal{U}=\{A\in {\rm{SA}}(H);\ \sigma(A)\cap (a,b)=\emptyset\}$ is Borel. But it is well-known that $\mathcal{U}$ is in fact  SRT-closed (see e.g., \cite[Theorem VIII.24 (a)]{ReedSimonI} or \cite[Lemma 1.6]{Simon95}). 
\end{proof}
Next we show the Borelness of $A\mapsto \sigma_{\rm{ess}}(A)$. Note however that $\mathcal{V}=\{A\in {\rm{SA}}(H); \sigma_{\rm{ess}}(A)\cap (a,b)=\emptyset \}$ is neither open nor closed. 
In fact $A\mapsto \sigma_{\rm{ess}}(A)$ behaves quite discontinuously (with respect to any compatible Polish topology on $\mathcal{F}(\mathbb{R})$):
\begin{proposition}
Let $K,L\in \mathcal{F}(\mathbb{R})$ be nonempty. Then there exists $\{A_n\}_{n=1}^{\infty}\subset {\rm{SA}}(H)$ and $A\in {\rm{SA}}(H)$ with the property that 
\[\sigma_{\rm{ess}}(A_n)=K\ (n\in \mathbb{N}),\ \ \sigma_{\rm{ess}}(A)=L,\ \ \ A_n\stackrel{n\to \infty}{\to} A\ \ \ \ {\rm{in}}\ {\rm{SA}}(H).\]
\end{proposition}
\begin{proof}For each $k\in \mathbb{N}$, let $H_k$ be a separable infinite-dimensional Hilbert space with CONS $\{\xi_{k,i}\}_{i=1}^{\infty}$, and let $e_{k,i}$ be the projection of $H_k$ onto $\mathbb{C}\xi_{k,i}\ (k,i\in \mathbb{N})$ Set $H=\bigoplus_{n=1}^{\infty}H_k$, and let $\{\lambda_n\}_{n=1}^{\infty}$ (resp. $\{\mu_n\}_{n=1}^{\infty}$) be a dense subset of $K$ (resp. $L$). 
For each $n\in \mathbb{N}$, define $A_{n,k}\in {\rm{SA}}(H_k)$ by 
\[A_{n,k}:=\begin{cases}
\displaystyle \mu_k\sum_{i=1}^{n}e_{k,i}+\lambda_k\sum_{i=n+1}^{\infty}e_{k,i} & (1\le k\le n)\\
\ \ \ \ \ \ \ \ \ \ \ \lambda_k1_{H_k} & (k>n).
\end{cases}\]
Define $A:=\bigoplus_{k=1}^{\infty}\mu_k1_{H_k}$, and $A_n:=\bigoplus_{k=1}^{\infty}A_{n,k}$. 
It holds that $\sigma_{\rm{ess}}(A)=\overline{\{\mu_k\}_{k=1}^{\infty}}=L$. 
On the other hand, $A_n$ is diagonalizable with eigenvalues $\{\lambda_k\}_{k=1}^{\infty}$ of infinite multiplicity and $\{\mu_1,\cdots,\mu_n\}$ of finite multiplicity (possibly some $\mu_i$ and $\lambda_j$ are equal). Therefore $\sigma_{\rm{ess}}(A_n)=\overline{\{\lambda_k\}_{k=1}^{\infty}}=K$. 
By construction, it holds that for each $k\in \mathbb{N}$, we have $A_{n,k}\stackrel{n\to \infty}{\to} \mu_k1_{H_k}$ (SRT). Therefore $A_n\stackrel{n\to \infty}{\to}A$ (SRT) holds. 
\end{proof}
\begin{theorem}\label{thm: taking ess spectrum is Borel}
The map $\Phi\colon {\rm{SA}}(H)\ni A\mapsto \sigma_{\rm{ess}}(A)\in \mathcal{F}(\mathbb{R})$ is Borel.
\end{theorem}
\begin{lemma}\label{lem: intersection of compact perturbations=essential}
Let $A\in {\rm{SA}}(H)$, and let $\mathcal{K}$ be a norm-dense subset of $\mathbb{K}(H)_{\rm{sa}}$. Then 
the following equality holds: 
\[\displaystyle \sigma_{\rm{ess}}(A)=\bigcap_{K\in \mathcal{K}}\sigma(A+K).\]
\end{lemma}
\begin{proof}
By Weyl's Theorem, the essential spectra are invariant under compact perturbations (Theorem \ref{thm: von Neumann theorem} (2)$\Rightarrow$(1)). Therefore
\[\sigma_{\rm{ess}}(A)=\bigcap_{K\in \mathbb{K}(H)_{\rm{sa}}}\sigma_{\rm{ess}}(A+K)\subset \bigcap_{K\in \mathbb{K}(H)_{\rm{sa}}}\sigma (A+K)\subset \bigcap_{K\in \mathcal{K}}\sigma(A+K).\]
To prove the opposite inclusion, we show that $\sigma_{{\rm{d}}}(A)\cap \bigcap_{K\in \mathcal{K}}\sigma(A+K)=\emptyset$. If $\sigma_{\rm{d}}(A)=\emptyset$, this is obvious, so we assume that $\sigma_{\rm{d}}(A)\neq \emptyset$. Let $E_A(\cdot)$ be the spectral measure of $A$, and let $\lambda \in \sigma_{\rm{d}}(A)$. Then by the definition of the discrete spectra, there exists $\delta>0$ such that $E_A((\lambda-\delta,\lambda+\delta))=E_A(\{\lambda\})$ has rank $n\in \mathbb{N}$. Put $K:=E_A(\{\lambda\})\in \mathbb{K}(H)_{\rm{sa}}$ (which is of finite rank). Then 
\[A+K-\lambda=E_A(\{\lambda\})+(A-\lambda)E_A(\mathbb{R}\setminus (\lambda-\delta,\lambda+\delta)).\] 
This shows that $A+K-\lambda$ has the bounded inverse $(A+K-\lambda)^{-1}$. 
Choose by density an element $K_0\in \mathcal{K}$ such that $\|K-K_0\|<1/\|(A+K-\lambda)^{-1}\|$. Then 
$\|(A+K-\lambda)^{-1}(K_0-K)\|<1$, and $1+(A+K-\lambda)^{-1}(K_0-K)$ is invertible in $\mathbb{B}(H)$. It holds that
\eqa{
A+K_0-\lambda&=A+K-\lambda+(K_0-K)\\
&=(A+K-\lambda)\{1+(A+K-\lambda)^{-1}(K_0-K)\}
}
has the bounded inverse $(A+K_0-\lambda)^{-1}=\{1+(A+K-\lambda)^{-1}(K_0-K)\}^{-1}(A+K-\lambda)^{-1}$. 
Therefore $\lambda\notin \sigma(A+K_0)$, and we obtain 
\[
\sigma_{\rm{d}}(A)\subset \mathbb{R}\setminus \bigcap_{K\in \mathcal{K}}\sigma(A+K).\]
This finishes the proof.  
\end{proof}
\begin{proof}[Proof of Theorem \ref{thm: taking ess spectrum is Borel}]
Let $\mathcal{K}=\{K_n\}_{n=1}^{\infty}$ be a countable norm-dense subset of $\mathbb{K}(H)_{\rm{sa}}$. By Lemma \ref{lem: intersection of compact perturbations=essential}, we have 
\begin{equation}
\sigma_{\rm{ess}}(A)=\bigcap_{n=1}^{\infty}\sigma(A+K_n),\ \ \ \ \ \ A\in {\rm{SA}}(H).\label{eq: countable intersection and ess}
\end{equation}
By Lemma \ref{lem: spectrum is Borel}, the map ${\rm{SA}}(H)\ni A\mapsto \sigma(A)\in \mathcal{F}(\mathbb{R})$ is Borel. To show that $\Phi:A\mapsto \sigma_{\rm{ess}}(A)$ is Borel, we have only to show that the set $\mathcal{B}:=\{A\in {\rm{SA}}(H);\ \sigma_{\rm{ess}}(A)\cap [a,b]\neq \emptyset\}$ is Borel for every closed interval $[a,b]$ in $\mathbb{R}$ (since every open set in a metrizable space is $F_{\sigma}$). 
Now we use the following equivalence:
\begin{equation}
\bigcap_{n=1}^{\infty}\sigma(A+K_n)\cap [a,b]\neq \emptyset\Leftrightarrow \bigcap_{n=1}^N\sigma(A+K_n)\cap [a,b]\neq \emptyset \ \ \ {\rm{for\ all}}\ N\in \mathbb{N}.\label{eq: compactness and ess spec}
\end{equation}
Indeed, $\Rightarrow$ is obvious, and $\Leftarrow$ follows from the finite intersection property of the compact set $[a,b]$. 
Then for each $A\in {\rm{SA}}(H)$, we deduce by (\ref{eq: countable intersection and ess}) and (\ref{eq: compactness and ess spec}) that
\eqa{
\mathcal{B}&=\left \{A\in {\rm{SA}}(H); \bigcap_{n=1}^{\infty}\sigma(A+K_n)\cap [a,b]\neq \emptyset \right \}\\
&=\bigcap_{k=1}^{\infty}\underbrace{\left \{ A\in {\rm{SA}}(H); \bigcap_{n=1}^{k}\sigma(A+K_n)\cap [a,b]\neq \emptyset \right \}}_{=:\mathcal{B}_k}
}
Therefore it is enough to show that $\mathcal{B}_k$ is Borel for each $k\in \mathbb{N}$. Recall that by Christensen's Theorem \ref{thm: Christensen}, the intersection map $I_2\colon \mathcal{F}(\mathbb{R})\times \mathcal{F}(\mathbb{R})\ni (K_1,K_2)\mapsto K_1\cap K_2\in \mathcal{F}(\mathbb{R})$ is Borel. By inductive argument, the $k$-fold intersection map $I_k\colon \mathcal{F}(\mathbb{R})^k\ni (K_1,\cdots,K_k)\mapsto \bigcap_{i=1}^kK_i\in \mathcal{F}(\mathbb{R})$ is Borel. 
Since the addition map $\tau_n\colon {\rm{SA}}(H)\ni A\mapsto A+K_n\in {\rm{SA}}(H)$ is a homeomorphism for each $n\in \mathbb{N}$, the Borelness of $\Phi_0:{\rm{SA}}(H)\ni A\mapsto \sigma(A)\in \mathcal{F}(\mathbb{R})$ implies that the map
\[\Psi_k:=I_k\circ (\Phi_0\circ \tau_1\times\cdots \times \Phi_0\circ \tau_k)\colon {\rm{SA}}(H)\ni A\mapsto \bigcap_{i=1}^k\sigma(A+K_i)\in \mathcal{F}(\mathbb{R})\]
is Borel. Thus $\mathcal{B}_k=\Psi_k^{-1}(\{K\in \mathcal{F}(\mathbb{R});\ K\cap [a,b]\neq \emptyset\})$ is Borel, so $\Phi:A\mapsto \sigma_{\rm{ess}}(A)$ is Borel.  
\end{proof}

\begin{proof}[Proof of Theorem \ref{thm: E_G^B(H)_sa is smooth}]
Let $A,B\in \mathbb{B}(H)_{\rm{sa}}$. By Weyl-von Neumann Theorem \ref{thm: von Neumann theorem}, $AE_G^{\mathbb{B}(H)_{\rm{sa}}}B$ if and only if $\sigma_{\rm{ess}}(A)=\sigma_{\rm{ess}}(B)$. 
Therefore $\Phi\colon {\rm{SA}}(H)\ni A\mapsto \sigma_{\rm{ess}}(A)\in \mathcal{F}(\mathbb{R})$ restricted to $\mathbb{B}(H)_{\rm{sa}}$ is a Borel reduction of $E_G^{\mathbb{B}(H)_{\rm{sa}}}$ to ${\rm{id}}_{\mathcal{F}(\mathbb{R})}$.  
\end{proof}
\subsection{Non-classification: Unbounded Case}\label{subsec: Turbulence: Unbounded Case}
We have shown that $E_G^{\mathbb{B}(H)_{\rm{sa}}}$ is smooth (Theorem \ref{thm: E_G^B(H)_sa is smooth}), therefore bounded self-adjoint operators are concretely classifiable up to Weyl-von Neumann equivalence by their essential spectra. 
In this section, we show that the situation for unbounded operators is rather different: we show that the $G$-action on ${\rm{SA}}(H)$ has a dense $G_{\delta}$ orbit ${\rm{SA}}_{{\rm{full}}}(H):=\{A\in {\rm{SA}}(H); \sigma_{\rm{ess}}(A)=\mathbb{R}\}$ (Theorem \ref{thm: unbounded case: comeager orbit exists}), whence the action is not generically turbulent. On the other hand we also show that it is unclassifiable by countable structures by showing that $E_G^Y\le_BE_G^{{\rm{SA}}(H)}$ for a generically turbulent $G$-action on a Polish space $Y$ (Theorem \ref{thm: action is weakly turbulent}). In fact we choose $Y={\rm{EES}}(H)=\{A\in {\rm{SA}}(H); \sigma_{\rm{ess}}(A)=\emptyset\}$, which is a small part of ${\rm{SA}}(H)$. $Y$ equipped with the norm resolvent topology is Polish (Proposition \ref{prop: EES with NRT is Polish}), and the $G$-action on $Y$ is just the restriction of the original action to $Y$. 
Note that $A\mapsto \sigma_{\rm{ess}}(A)$ is constant on $Y$, and therefore $\sigma_{\rm{ess}}(\cdot)$ is very far from a complete invariant for $E_G^{{\rm{SA}}(H)}$.  
\subsubsection{$E_G^{{\rm{SA}}(H)}$ is Not Generically Turbulent}
We show that the $G$-action on ${\rm{SA}}(H)$ is not generically turbulent, by showing that there exists a comeager $G$-orbit. More precisely:
\begin{theorem}\label{thm: unbounded case: comeager orbit exists}
The following statements hold:
\begin{list}{}{}
\item[{\rm{(1)}}] The set ${\rm{SA}}_{{\rm{full}}}(H):=\{A\in {\rm{SA}}(H); \sigma_{\rm{ess}}(A)=\mathbb{R}\}$ is a dense $G_{\delta}$ subset of ${\rm{SA}}(H)$. 
\item[{\rm{(2)}}] If $A,B\in {\rm{SA}}(H)$ satisfy $\sigma_{\rm{ess}}(A)=\sigma_{\rm{ess}}(B)=\mathbb{R}$, then $AE_G^{{\rm{SA}}(H)}B$. Consequently,  
${\rm{SA}}_{{\rm{full}}}(H)$ is a dense $G_{\delta}$-orbit of the $G$-action. 
\end{list}
In particular, the $G$-action on ${\rm{SA}}(H)$ is not generically turbulent. 
\end{theorem}
Note that the above theorem shows that (1)$\Rightarrow $(2) of Theorem \ref{thm: von Neumann theorem} holds true for ${\rm{SA}}_{\rm{full}}(H)$ (in fact von Neumann's proof itself works verbatim, as we see below), even though elements in ${\rm{SA}}_{\rm{full}}(H)$ are highly unbounded.
The proof of Theorem \ref{thm: unbounded case: comeager orbit exists} (1) is strongly inspired by the work of B. Simon \cite{Simon95}.  We start from the next result whose proof is divided into a series of elementary lemmata:  
\begin{proposition}\label{prop: essential spec contains lamda is dnese G_delta}
Let $\lambda \in \mathbb{R}$. Then the set $\{A\in {\rm{SA}}(H);\ \lambda \in \sigma_{\rm{ess}}(A)\}$ is a dense $G_{\delta}$ set in ${\rm{SA}}(H)$.
\end{proposition}
\begin{lemma}\label{lem: close to independent vectors is independent}
Let $k\in \mathbb{N}$, and let $\xi_1,\cdots, \xi_k\in H$ be linearly independent vectors. Then there exists $\varepsilon>0$ such that whenever $\eta_1,\cdots, \eta_k\in H$ satisfy $\|\xi_i-\eta_i\|<\varepsilon\ (1\le i\le k)$, then $\eta_1,\cdots ,\eta_k$ are also linearly independent.
\end{lemma}
\begin{proof}
Assume by contradiction that there exist sequences $\{\eta_j^{(l)}\}_{l=1}^{\infty}\ (1\le j\le k)$ such that $\eta_j^{(l)}\stackrel{l\to \infty}{\to}\xi_j\ (1\le j\le k)$ and $(\eta_1^{(l)},\cdots,\eta_k^{(l)})$ are not linearly independent, whence for each $l\in \mathbb{N}$, there exists $(\lambda_1^{(l)},\cdots,\lambda_k^{(l)})\in \mathbb{C}^k\setminus \{(0,\cdots ,0)\}$ such that 
$\sum_{j=1}^k\lambda_j^{(l)}\eta_j^{(l)}=0$. We may assume (by rescaling) that $\max_{1\le j\le k}|\lambda_j^{(l)}|=1$ for each $l\in \mathbb{N}$. Put $I_j:=\{l\in \mathbb{N};|\lambda_j^{(l)}|=1\}$. Then as $\mathbb{N}=\bigcup_{j=1}^kI_j$, there exists $k_0\in \{1,\cdots,k\}$ such that $I_{k_0}$ is an infinite set. Enumerate $I_{k_0}=\{l_n\}_{n=1}^{\infty}\ (l_1<l_2<\cdots )$. Then by compactness, we may find a subsequence in $I_{k_0}$, still denoted as $\{l_n\}_{n=1}^{\infty}$, such that $\lambda_j^{(l_n)}$ converges to some $\lambda_j\ (1\le j\le k)$. In particular, we have
\[\sum_{j=1}^k\lambda_j\xi_j=\lim_{n\to \infty}\sum_{j=1}^{k}\lambda_j^{(l_n)}\eta_j^{(l_n)}=0.\]
Since $\xi_1,\cdots, \xi_k$ are linearly independent, we have $(\lambda_1,\cdots,\lambda_k)=(0,\cdots,0)$. This is a contradiction, as $|\lambda_{k_0}|=\lim_{n\to \infty}|\lambda_{k_0}^{(l_n)}|=1\neq 0$.  
\end{proof}
\begin{lemma}\label{lem: strong limit does not increase rank}
Let $\{a_n\}_{n=1}^{\infty}\subset \mathbb{B}(H)$ be a sequence converging strongly to $a\in \mathbb{B}(H)$. If ${\rm{rank}}(a_n)\le k\ (n\in \mathbb{N})$ holds for some fixed $k\in \mathbb{N}$, then ${\rm{rank}}(a)\le k$ holds. 
\end{lemma}
\begin{proof}
If $m:={\rm{rank}}(a)$ is $0$, there is nothing to prove. If $m\ge 1$, there exist $\xi_1,\cdots,\xi_m\in H$ such that $a\xi_1,\cdots, a\xi_m$ are linearly independent. 
Since $a_n\xi_j\stackrel{n\to \infty}{\to} a\xi_j\ (1\le j\le m)$, by Lemma \ref{lem: close to independent vectors is independent} there exists $N\in \mathbb{N}$ such that $a_N\xi_1,\cdots,a_N\xi_m$ are linearly independent, whence $m\le {\rm{rank}}(a_N)\le k$. 
\end{proof}
\begin{lemma}\label{lem: rank of indicator function less than k}
Let $A\in {\rm{SA}}(H)$, $(a,b)$ be an open interval in $\mathbb{R}$, and let $k\in \mathbb{N}$. Then ${\rm{rank}}E_A((a,b))\le k$ holds if and only if for every continuous real valued function $f$ on $\mathbb{R}$ with ${\rm{supp}}(f)\subset (a,b)$, ${\rm{rank}}f(A)\le k$ holds.  
\end{lemma}
\begin{proof}
$(\Rightarrow)$ Assume that ${\rm{rank}}E_A((a,b))=n\le k$. Then $AE_A((a,b))=\sum_{i=1}^n\lambda_ie_i$, where $\lambda_i\in (a,b)\ (1\le i\le n)$ and $\{e_i\}_{i=1}^n$ is a pairwise orthogonal projections of rank 1. Then for any continuous function $f$ with ${\rm{supp}}(f)\subset (a,b)$, $f(A)=\sum_{i=1}^nf(\lambda_i)e_i$ has rank less than or equal to $k$. \\
$(\Leftarrow)$ We prove the contrapositive: assume that $m:={\rm{rank}}E_A((a,b))\ge k+1$. Since $E_A((a+\frac{1}{n},b-\frac{1}{n}))\stackrel{n\to \infty}{\to}E_A((a,b))$ (SOT), there is a closed subinterval $[\alpha,\beta]\subsetneq (a,b)$ with ${\rm{rank}}E_A([\alpha,\beta])=m$. Let $f$ be a continuous positive valued function whose support contained in $(a,b)$ such that $f|_{[\alpha,\beta]}=1$. Then ${\rm{rank}}f(A)=m\ge k+1$. 
\end{proof}
\begin{proof}[Proof of Proposition \ref{prop: essential spec contains lamda is dnese G_delta}]
We express the complement of $\{A\in {\rm{SA}}(H);\ \lambda\in \sigma_{\rm{ess}}(A)\}$ as follows:
\eqa{
\left \{A\in {\rm{SA}}(H);\ \lambda \notin \sigma_{\rm{ess}}(A)\right \}&=\bigcup_{\varepsilon>0}\left \{A\in {\rm{SA}}(H);\ E_A((\lambda-\varepsilon,\lambda+\varepsilon)){\rm{\ is\ of\ finite\ rank}}\ \right \}\\
&=\bigcup_{n=1}^{\infty}\bigcup_{k=1}^{\infty}\underbrace{\{A\in {\rm{SA}}(H); {\rm{rank}}E_A((\lambda-\tfrac{1}{n},\lambda+\tfrac{1}{n}))\le k\}}_{=:S_{n,k}}
}
We show that $S_{n,k}$ is closed in ${\rm{SA}}(H)$. Suppose $\{A_m\}_{m=1}^{\infty}\subset S_{n,k}$ converges to $A\in {\rm{SA}}(A)$. Then let $f$ be a continuous function with ${\rm{supp}}(f)\subset (\lambda-\frac{1}{n},\lambda+\frac{1}{n})$. 
Then as $A_m\stackrel{m\to \infty}{\to}A$ (SRT), we have $f(A_m)\stackrel{m\to \infty}{\to}f(A)$ (SOT) by \cite[Theorem VIII.20 (b)]{ReedSimonI}. By Lemma \ref{lem: rank of indicator function less than k}, ${\rm{rank}}f(A_m)\le k$ for each $m\in \mathbb{N}$. Therefore by Lemma \ref{lem: strong limit does not increase rank}, ${\rm{rank}}f(A)\le k$. Since this holds for arbitrary such $f$, Lemma \ref{lem: rank of indicator function less than k} shows that ${\rm{rank}}E_A((\lambda-\tfrac{1}{n},\lambda+\tfrac{1}{n}))\le k$. Therefore $A\in S_{n,k}$. This shows that $\left \{A\in {\rm{SA}}(H);\ \lambda \notin \sigma_{\rm{ess}}(A)\right \}$ is $F_{\sigma}$, so its complement $\{A\in {\rm{SA}}(H);\ \lambda\in \sigma_{\rm{ess}}(A)\}$ is $G_{\delta}$. Next we show the density. Let $A\in {\rm{SA}}(H)$, and let $\mathcal{V}$ be an open neighborhood of $A$. Then by Weyl-von Neumann Theorem \ref{thm: Weyl-von Neumann}, we may find $A_0\in \mathcal{V}$ of the form $A_0=\sum_{n=1}^{\infty}\lambda_ne_n$, where $\{e_n\}_{n=1}^{\infty}$ is a mutually orthogonal projections with sum equal to 1, and $\lambda_n\in \mathbb{R}$. Then put $A_k:=\sum_{n=1}^k\lambda_ne_n+\lambda\sum_{n=k+1}^{\infty}e_n$. Clearly $\lambda \in \sigma_{\rm{ess}}(A_k)$. Furthermore, as $q_k=\sum_{n=k+1}^{\infty}e_n$ tends strongly to 0 as $k\to \infty$, we have
\[
(A_k-i)^{-1}-(A_0-i)^{-1}=\sum_{n=k+1}^{\infty}\left (\frac{1}{\lambda-i}-\frac{1}{\lambda_n-i}\right )e_n\stackrel{k\to \infty}{\to}0\ ({\rm{SOT}}).
\]
Therefore the closure of the set $\{A\in {\rm{SA}}(H);\ \lambda\in \sigma_{\rm{ess}}(A)\}$ intersects $\mathcal{V}$, whence $\{A\in {\rm{SA}}(H);\ \lambda\in \sigma_{\rm{ess}}(A)\}$ is a dense $G_{\delta}$ set. 
\end{proof}
\begin{proof}[Proof of Theorem \ref{thm: unbounded case: comeager orbit exists} (1)]
For each $q\in \mathbb{Q}$, the set $G_q:=\{A\in {\rm{SA}}(H);\ q\in \sigma_{\rm{ess}}(A)\}$ is a dense $G_{\delta}$ set in ${\rm{SA}}(H)$ by Proposition \ref{prop: essential spec contains lamda is dnese G_delta}. Therefore as $\sigma_{\rm{ess}}(A)$ is a closed subset in $\mathbb{R}$,
\[\{A\in {\rm{SA}}(H);\ \sigma_{\rm{ess}}(A)=\mathbb{R}\}=\bigcap_{q\in \mathbb{Q}}G_q\]
is also a dense $G_{\delta}$ set.  
\end{proof}
To finish the proof of Theorem \ref{thm: unbounded case: comeager orbit exists} (2), we need the following variant of a well-known argument used in the proof of Theorem \ref{thm: von Neumann theorem}, (1)$\Rightarrow $(2). 
\begin{lemma}[\cite{AkhiezerGlazman}]\label{lem: permutation lemma}
Let $\{\lambda_n\}_{n=1}^{\infty},\{\mu_n\}_{n=1}^{\infty}$ be sequences of real numbers with the same set of accumulation point $M$. If both $\{\lambda_n\}_{n=1}^{\infty},\{\mu_n\}_{n=1}^{\infty}$ have only finitely many isolated points, then there exists a permutation $\pi$ of $\mathbb{N}$ such that $\lim_{n\to \infty}(\lambda_n-\mu_{\pi(n)})=0$ holds. 
\end{lemma}
\begin{proof}The setting as well as the proof is almost the same as the one in \cite[$\S$94]{AkhiezerGlazman}, so we only explain the difference of the present setting from \cite[$\S$94]{AkhiezerGlazman}. 
For $k\in \mathbb{N}$, define 
\[\varepsilon_k:=\inf_{t\in M}|\lambda_k-t|+\frac{1}{k},\ \ \ \eta_k:=\inf_{t\in M}|\mu_k-t|+\frac{1}{k}.\]
Since there are only finitely many isolated points in $\{\lambda_n\}_{n=1}^{\infty},\{\mu_n\}_{n=1}^{\infty}$, all but finitely many members of $\{\mu_n\}_{n=1}^{\infty},\ \{\lambda_n\}_{n=1}^{\infty}$ belong to $M$. Therefore $\varepsilon_k\to 0,\ \eta_k\to 0\ (k\to \infty)$.  Now the rest of the proof is the same as the one in \cite[$\S$94]{AkhiezerGlazman}, so we omit the proof.
\end{proof}
\begin{remark}
Note that Lemma \ref{lem: permutation lemma} does not hold in general without assuming some conditions on isolated points of $M$: consider the sequences $\{\lambda_n^{(t)}\}_{n=1}^{\infty}\ (t\in [0,1])$ in Example \ref{ex: non-unitarily equivalent but have same ess spec}. We show that if $0\le s<t\le 1$, then there is no permutation $\pi$ of $\mathbb{N}$ satisfying $\lim_{n\to \infty}(\lambda_{\pi(n)}^{(s)}-\lambda_{n}^{(t)})=0$, although both sequences have accumulation points $M=\mathbb{N}$. Assume by contradiction that such $\pi$ exists. Then there exists $k_0\in \mathbb{N}$ such that $|\lambda_{\pi(k)}^{(s)}-\lambda_k^{(t)}|<\frac{t-s}{8}$ for all $k\ge k_0$. Since $k\le 2^{k}\le 2^{k}(2m-1)\stackrel{\text{def}}{=}\nai{k+1}{m}\ (k,m\in \mathbb{N})$, this implies that
\begin{equation}
|\lambda_{\pi(\nai{k}{m})}^{(s)}-\lambda_{\nai{k}{m}}^{(t)}|<\frac{t-s}{8},\ \ \ \ \ \ \ \ (k\ge k_0+1,m\in \mathbb{N}).\label{eq: lambda^s and lambda^t close 2}
\end{equation}
On the other hand, if $k,k',m,m'\in \mathbb{N}$ with $k\neq k', k,k'\ge k_0+1$, then 
\eqa{
|\lambda_{\nai{k'}{m'}}^{(s)}-\lambda_{\nai{k}{m}}^{(t)}|&=\left |k'+\frac{s}{m'+2}-k-\frac{t}{m+2}\right |\\
&\ge |k'-k|-\left |\frac{s}{m'+2}-\frac{t}{m+2}\right |\\
&\ge 1-\frac{2}{3}>\frac{t-s}{8}.
}
Therefore by (\ref{eq: lambda^s and lambda^t close 2}), for each $k\ge k_0+1$ and $m\in \mathbb{N}$, there exists $\varphi_k(m)\in \mathbb{N}$ such that 
\begin{equation}
\lambda_{\pi(\nai{k}{m})}^{(s)}=\lambda_{\nai{k}{\varphi_k(m)}}^{(s)} \ \ \ \ \ \ \ (k\ge k_0+1,m\in \mathbb{N}).\label{eq: lambda^s vs lambda^t contradiction}
\end{equation}
However, by (\ref{eq: lambda^s and lambda^t close 2}) and (\ref{eq: lambda^s vs lambda^t contradiction}) (for $k=k_0+1,m=1$) it follows that
\eqa{
\frac{t-s}{8}&>|\lambda_{\nai{k_0+1}{\varphi_{k_0+1}(1)}}^{(s)}-\lambda_{\nai{k_0+1}{1}}^{(t)}|=\left |\frac{s}{\varphi_{k_0+1}(1)+2}-\frac{t}{3}\right |\\
&\ge \frac{t-s}{3},
}
which is a contradiction. This completes the proof. 
\end{remark}
\begin{proof}[Proof of Theorem \ref{thm: unbounded case: comeager orbit exists} (2)] 
The proof goes exactly the same as von Neumann's proof: by Weyl-von Neumann Theorem \ref{thm: Weyl-von Neumann}, $[A]_G$, $[B]_G$ contain diagonalizable operators with essential spectra $\mathbb{R}$. Therefore we may assume that $A$, $B$ are of the form $A=\sum_{n=1}^{\infty}a_n\nai{\xi_n}{\ \cdot\ }\xi_n,\ B=\sum_{n=1}^{\infty}b_n\nai{\eta_n}{\ \cdot\ }\eta_n$, where $\{a_n\}_{n=1}^{\infty},\{b_n\}_{n=1}^{\infty}$ are real sequences. Since $\sigma_{\rm{ess}}(A)=\sigma_{\rm{ess}}(B)=\mathbb{R}$ and there are at most countably many isolated eigenvalues, this implies that the set of accumulation points of $\{a_n\}_{n=1}^{\infty}$ and $\{b_n\}_{n=1}^{\infty}$ are both $\mathbb{R}$. By Lemma \ref{lem: permutation lemma}, there exists a permutation $\pi$ of $\mathbb{N}$ such that $\lim_{k\to \infty}(a_{\pi(k)}-b_k)=0$. Define $u\in \mathcal{U}(H)$ by
\[u\xi_k:=\eta_{\pi^{-1}(k)},\ \ \ \ \ k\in \mathbb{N}.\]
Then $uAu^*=\sum_{n=1}^{\infty}a_{\pi(n)}\nai{\eta_n}{\ \cdot\ }\eta_n$, and define $K:=\sum_{n=1}^{\infty}(b_n-a_{\pi(n)})\nai{\eta_n}{\ \cdot\ }\eta_n\in \mathbb{K}(H)_{\rm{sa}}$. 
It holds that $uAu^*+K=B$.  
\end{proof}

\subsubsection{$E_G^{{\rm{EES}}(H)}$ is Generically Turbulent}
As explained in the introduction to $\S$\ref{subsec: Turbulence: Unbounded Case}, we now study the restricted action of $G$ on a subset ${\rm{EES}}(H)=\{A\in {\rm{SA}}(H);\sigma_{\rm{ees}}(A)=\emptyset\}$ equipped with a new Polish topology. 
\paragraph{Norm Resolvent Topology and Polish Space ${\rm{EES}}(H)$}\label{para: EES, NRT is Polish}\ \\
Let ${\rm{EES}}(H):=\{A\in {\rm{SA}}(H);\ \sigma_{\rm{ess}}(A)=\emptyset\}$  (EES stands for Empty Essential Spectrum). 
\begin{definition}
The  {\it norm resolvent topology} (NRT) on ${\rm{SA}}(H)$ is the weakest topology for which the map ${\rm{SA}}(H)\ni A\mapsto (A-i)^{-1}\in \mathbb{B}(H)$ is norm-continuous.
\end{definition}
Note that $({\rm{SA}}(H),{\rm{NRT}})$ is not separable, whence not Polish: to see this, choose a CONS $\{\xi_n\}_{n=1}^{\infty}$ for $H$ and define for each $F\subset \mathbb{N}$ an operator $A_F$ by
\[A_F:=\sum_{n=1}^{\infty}1_F(n)\nai{\xi_n}{\ \cdot\ }\xi_n\in {\rm{SA}}(H),\]
where $1_F$ is the characteristic function on $F$.  
Then $\{A_F\}_{F\subset \mathbb{N}}$ is an uncountable family, and 
\[\|(A_{F_1}-i)^{-1}-(A_{F_2}-i)^{-1}\|=\sup_{n\ge 1}|(1_{F_1}(n)-i)^{-1}-(1_{F_2}(n)-i)^{-1}|=\frac{1}{\sqrt{2}}\text{\ \ \ if\ }F_1\neq F_2.\]
Note that none of them are in ${\rm{EES}}(H)$. 
On the contrary, we show that 
\begin{proposition}\label{prop: EES(H) is NRT-Polish}
$({\rm{EES}}(H),{\rm{NRT}})$ is a Polish $G$-space with respect to the restriction $\beta$ of the action  $\alpha\colon G\curvearrowright {\rm{SA}}(H)$ to ${\rm{EES}}(H)$.
\end{proposition} 
We first show
\begin{proposition}\label{prop: EES with NRT is Polish}
$({\rm{EES}}(H),{\rm{NRT}})$ is a Polish space.
\end{proposition}
We need preparations. The first lemma is well-known. 
\begin{lemma}\label{lem: ess=empty iff compact resolvent}
Let $A\in {\rm{SA}}(H)$. Then $\sigma_{\rm{ess}}(A)=\emptyset$ if and only if $(A-i)^{-1}\in \mathbb{K}(H)$. 
\end{lemma}
\begin{proof}
Since $(A-i)^{-1}$ is normal, by spectral theory $(A-i)^{-1}\in \mathbb{K}(H)$, if and only if $(A-i)^{-1}$ is diagonalizable with eigenvalues $\{\frac{1}{a_n-i}\}_{n=1}^{\infty}\ (a_n\in \mathbb{R})$ satisfying $|1/(a_n-i)|\stackrel{n\to \infty}{\to}0$, if and only if $A$ is diagonalizable with eigenvalues $\{a_n\}_{n=1}^{\infty}$ satisfying $|a_n|\stackrel{n\to \infty}{\to}\infty$, i.e., $A\in {\rm{EES}}(H)$. 
\end{proof}
\begin{lemma}\label{lem: range of the resolvent map}
Let $x\in \mathbb{B}(H)$ be normal. Then there exists $A\in {\rm{SA}}(H)$ such that $x=(A-i)^{-1}$ holds, if and only if both ${\rm{Ran}}(x)$ and ${\rm{Ran}}(x^*)$ are dense in $H$, and $(x^{-1}+i)^*=x^{-1}+i$.
\end{lemma}
\begin{proof}
($\Rightarrow $) Suppose $x=(A-i)^{-1}$ for $A\in {\rm{SA}}(H)$. Then ${\rm{ker}}(x)=\{0\}={\rm{Ran}}(x^*)^{\perp}$, so ${\rm{Ran}}(x^*)$ is dense in $H$. Also, ${\rm{Ran}}(x)=\dom{A-i}=\dom{A}$ is dense in $H$. Since $x^{-1}=A-i$, so that $x^{-1}+i=A=A^*=(x^{-1}+i)^*$ holds.\\
($\Leftarrow $) By assumption, $\text{ker}(x)=\text{ker}(x^*)=\{0\}$, and $x^{-1}$ is densely defined. Then by assumption, $A:=x^{-1}+i$ is self-adjoint, and $x=(A-i)^{-1}$.  
\end{proof}
\begin{lemma}\label{lem: range of resolvents is Polish}
Let $\mathcal{D}$ be a subspace of $\mathbb{K}(H)$ consisting of those normal elements $x$ such that ${\rm{Ran}}(x)$ and ${\rm{Ran}}(x^*)$ are both dense in $H$. Then $\mathcal{D}$ is a $G_{\delta}$ subset of $\mathbb{K}(H)$ with respect to the norm topology. In particular, $\mathcal{D}$ is Polish.
\end{lemma}
\begin{proof}
It is clear that $\mathcal{D}_1:=\{x\in \mathbb{K}(H); xx^*=x^*x\}$ is closed. Let $\{\xi_n\}_{n=1}^{\infty}$ be a dense subset of $H$. Then it is easy to see that 
\[\textstyle {\rm{Ran}}(x){\rm{\ is\ dense\ }}\Leftrightarrow \forall k\in \mathbb{N}\ \forall l\in \mathbb{N}\ \exists m\in \mathbb{N}\ \ [\ \|x\xi_m-\xi_l\|<\frac{1}{k}\ ].\]
Therefore
\eqa{
\mathcal{D}_2:=\{x\in \mathbb{K}(H);\ {\rm{Ran}}(x){\rm{\ is\ dense\ }}\}&=\bigcap_{k=1}^{\infty}\bigcap_{l=1}^{\infty}\bigcup_{m=1}^{\infty}
\underbrace{\{x\in \mathbb{K}(H);\ \|x\xi_m-\xi_l\|<\tfrac{1}{k}\}}_{{\rm{open}}},
}
which is $G_{\delta}$ in $\mathbb{K}(H)$. Similarly, $\mathcal{D}_3:=\{x\in \mathbb{K}(H); {\rm{Ran}}(x^*)\}$ is $G_{\delta}$ in $\mathbb{K}(H)$, and so is $\mathcal{D}=\mathcal{D}_1\cap \mathcal{D}_2\cap \mathcal{D}_3$.
\end{proof}

\begin{proof}[Proof of Proposition \ref{prop: EES with NRT is Polish}]
Let $\varphi\colon ({\rm{ESS}}(H),{\rm{NRT}})\to (\mathbb{K}(H),\|\cdot \|)$ be a map given by $\varphi(A)=(A-i)^{-1},\ (A\in {\rm{SA}}(H))$. By the definition of NRT and the injectivity, $\varphi$ is a homeomorphism of ${\rm{ESS}}(H)$ onto its range. We see that
\[\varphi({\rm{EES}}(H))=\mathcal{D}_0:=\{x\in \mathcal{D};\ x^{-1}+i\in {\rm{SA}}(H)\},\]
where $\mathcal{D}$ is as in Lemma \ref{lem: range of resolvents is Polish}. 
Indeed, if $x=\varphi(A)\in \varphi({\rm{EES}}(H))$, then $x$ is compact by Lemma \ref{lem: ess=empty iff compact resolvent}. Moreover, ${\rm{Ran}}(x)$ and ${\rm{Ran}}(x^*)$ are dense in $H$, and $x^{-1}+i\in {\rm{SA}}(H)$ by Lemma \ref{lem: range of the resolvent map}. This shows that $x\in \mathcal{D}_0$. Conversely, if $x\in \mathcal{D}$ is such that $A:=x^{-1}+i\in {\rm{SA}}(H)$, then by Lemma \ref{lem: ess=empty iff compact resolvent}, $A\in {\rm{EES}}(H)$, and $\varphi(A)=x$. This shows that $\varphi({\rm{EES}}(H))=\mathcal{D}_0$.

We next show that $\mathcal{D}_0$ is closed in $\mathcal{D}$. Once this is proved, Lemma \ref{lem: range of resolvents is Polish} implies that $\mathcal{D}_0$ is also Polish, and so is ${\rm{EES}}(H)$. Let $\{x_n\}_{n=1}^{\infty}$ be a sequence in $\mathcal{D}_0$ converging in norm to $x\in \mathcal{D}$. Put $A_n:=x_n^{-1}+i\in {\rm{SA}}(H)$. Then $x_n=(A_n-i)^{-1}\stackrel{\|\cdot \|}{\longrightarrow }x,\ \ \ x_n^*=(A_n+i)^{-1}\stackrel{\|\cdot \|}{\longrightarrow }x^*$.  
Since $x\in \mathcal{D}$, $x$ and $x^*$ has dense ranges, whence by Kato-Trotter Theorem \cite[Theorem VIII.22]{ReedSimonI}, there exists $A\in {\rm{SA}}(H)$ such that $(A_n-i)^{-1}\stackrel{{\rm{SOT}}}{\to}(A-i)^{-1}$. Since $(A_n-i)^{-1}\stackrel{{\rm{SOT}}}{\to} x$ also, we have $x=(A-i)^{-1}$ and $x^{-1}+i=A\in {\rm{EES}}(H)$. Therefore $\mathcal{D}_0$ is closed in $\mathcal{D}$. This finishes the proof. 
\end{proof}
We now show that ${\rm{EES}}(H)$ is a Polish $G$-space. 
\begin{proposition}\label{prop: G on EES is continuous}
The action $\beta \colon G\times {\rm{EES}}(H)\to {\rm{EES}}(H)$ is continuous.
\end{proposition}
We need preparations. The proof of the next lemma is almost identical to that of Proposition \ref{prop: B(H)_sa acts on SA(H) continuously} (simply drop $\xi$ in (\ref{eq: Arai Neumann series}), (\ref{eq: Arai estimate 1 with xi}) and (\ref{eq: Arai estimate 2 with xi}) in the proof of Lemma \ref{lem: Arai's lemma} and use the joint norm-continuity of the operator product to get NRT-version of Lemma \ref{lem: Arai's lemma}). Therefore we omit the poof.
\begin{lemma}\label{lem: Arai Lemma for NRT}
Let $A_n,A\in {\rm{SA}}(H)$ and let $K_n,K\in \mathbb{B}(H)_{\rm{sa}}\ (n\in \mathbb{N})$. If  $A_n\stackrel{{\rm{NRT}}}{\rightarrow }A$ and $\|K_n-K\|\stackrel{n\to \infty}{\to}0$, then $A_n+K_n\stackrel{{\rm{NRT}}}{\rightarrow }A+K$ holds. 
\end{lemma}
The next lemma is known in operator theory. We add the proof for completeness.
\begin{lemma}\label{lem: norm convergence of compacts and sot unitary}
Let $x_n,x\in \mathbb{K}(H)$ and $u_n,u\in \mathcal{U}(H)\ (n\in \mathbb{N})$ be such that $\|x_n-x\|\stackrel{n\to \infty}{\to}0$ and $u_n\stackrel{n\to \infty}{\to}u$ {\rm{(SOT)}}. Then we have $\|u_nx_n-ux\|\stackrel{n\to \infty}{\to}0$.
\end{lemma}
\begin{proof}
 Assume by contradiction that $u_nx_n-ux$ does not converge to 0 in norm. Then we may find a subsequence $n_1<n_2<\cdots$ and $\varepsilon>0$ such that $\|u_{n_k}x_{n_k}-ux\|>\varepsilon$ for each $k\in \mathbb{N}$. Therefore for each $k\in \mathbb{N}$, there exists a unit vector $\xi_k\in H$ such that
\begin{equation}
\|u_{n_k}x_{n_k}\xi_k-ux\xi_k\|>\varepsilon,\ \ \ \ \ k\in \mathbb{N}. \label{eq: estimate in xi_k}
\end{equation}
Since the unit ball of $H$ is weakly compact and $H$ is separable, there is a subsequence $k_1<k_1<\cdots$ such that $\xi_{k_p}$ converges weakly to some $\xi\in H$ as $p\to \infty$. Then it holds that 
\eqa{
\|x_{n_{k_p}}\xi_{k_p}-x\xi\|&\le \|x_{n_{k_p}}\xi_{k_p}-x\xi_{k_p}\|+\|x\xi_{k_p}-x\xi\|\\
&\le \|x_{n_{k_p}}-x\|+\|x(\xi_{k_p}-\xi)\|.
}
In the right hand side of the last inequality, the first term tends to 0 as $p\to \infty$. Since $\xi_{k_p}-\xi$ converges weakly to 0, and $x$ is compact, the second term also tends to 0. 
Therefore we have
\begin{equation}
\lim_{p\to \infty}\|x_{n_{k_p}}\xi_{k_p}-x\xi\|=0.\label{eq: x_k_pxi_p tends to xxi}
\end{equation}
It then follows that
\eqa{
\|u_{k_p}x_{n_{k_p}}\xi_{k_p}-ux\xi\|&\le \|u_{n_{k_p}}x_{n_{k_p}}\xi_{k_p}-u_{n_{k_p}}x\xi\|+\|(u_{n_{k_p}}-u)x\xi\|\\
&\le \|x_{n_{k_p}}\xi_{k_p}-x\xi\|+\|(u_{n_{k_p}}-u)x\xi\|,
}
and by $u_n\stackrel{{\rm{SOT}}}{\to}u$ and Eq. (\ref{eq: x_k_pxi_p tends to xxi}), we obtain
\begin{equation}
\lim_{p\to \infty}\|u_{k_p}x_{n_{k_p}}\xi_{k_p}-ux\xi\|=0. \label{eq: u_n_k_px_n_k_pxi_k_p converes to uxxi}
\end{equation}
Combining  (\ref{eq: u_n_k_px_n_k_pxi_k_p converes to uxxi}) 
and $\|x\xi_{k_p}-x\xi\| \stackrel{p\to \infty}{\to}0$, we obtain
\eqa{
\|u_{n_{k_p}}x_{n_{k_p}}\xi_{k_p}-ux\xi_{k_p}\|&\le \|u_{n_{k_p}}x_{n_{k_p}}\xi_{k_p}-ux\xi\|+\|ux\xi-ux\xi_{k_p}\|\\
&\stackrel{p\to \infty}{\rightarrow }0,
}
contradicting (\ref{eq: estimate in xi_k}). Therefore $\|u_nx_n-ux\|\stackrel{n\to \infty}{\to}0$ holds. 
\end{proof}
\begin{proof}[Proof of Proposition \ref{prop: G on EES is continuous}]
Let $u_n,u\in \mathcal{U}(H)$, $K_n,K\in \mathbb{K}(H)_{\rm{sa}}$, and $A_n,A\in {\rm{EES}}(H)\ (n\in \mathbb{N})$ be such that $u_n\stackrel{{\rm{SOT}}}{\to}u$, $A_n\stackrel{\rm{NRT}}{\to}A$ and $K_n\stackrel{\|\cdot\|}{\to}K$, respectively. We show that $u_nA_nu_n^*+K_n\stackrel{{\rm{NRT}}}{\to}uAu^*+K$. By Lemma \ref{lem: Arai Lemma for NRT}, it suffices to prove that $u_nA_nu_n^*\stackrel{\rm{NRT}}{\to}uAu^*$. We compute the resolvent as follows:
\begin{align}
\|(u_nA_nu_n^*-i)^{-1}-(uAu^*-i)^{-1}\|&=\|u_n(A_n-i)^{-1}u_n^*-u(A-i)^{-1}u^*\| \notag \\
&\le \|\{u_n(A_n-i)^{-1}-u(A-i)^{-1}\}u_n^*\| \notag \\
&\hspace{2.0cm}+\|u\{(A-i)^{-1}u^*-(A-i)^{-1}u_n^*\}\| \notag \\
&=\|u_n(A_n-i)^{-1}-u(A-i)^{-1}\| \notag \\
&\hspace{2.0cm}+\|u(A+i)^{-1}-u_n(A+i)^{-1}\|.\label{eq: u_nA_nu_n*}
\end{align}
Since $(A_n-i)^{-1},\ (A\pm i)^{-1}$ are compact (Lemma \ref{lem: ess=empty iff compact resolvent}), the assumptions on $u_n$ and $A_n$ implies that (\ref{eq: u_nA_nu_n*}) converges to 0 by Lemma \ref{lem: norm convergence of compacts and sot unitary}. Therefore $u_nA_nu_n^*\stackrel{\rm{NRT}}{\to}uAu^*$. This finishes the proof. 
\end{proof}
\paragraph{Generic Turbulence}\ \\
Finally, we show the generic turbulence of $G\curvearrowright {\rm{EES}}(H)$. 
\begin{theorem}\label{thm: action is weakly turbulent}
The restricted action $\beta$ of $G$ on ${\rm{EES}}(H)$ is generically turbulent.
\end{theorem}
Before we prove Theorem \ref{thm: action is weakly turbulent}, let us state an immediate consequence:
\begin{theorem}\label{thm: E_G^SA is unclassifiable by countable structures}
$E_G^{{\rm{SA}}(H)}$ does not admit classification by countable structures.
\end{theorem}
\begin{proof}
By Theorem \ref{thm: action is weakly turbulent}, $E_G^{{{\rm{EES}}(H)}}$ is generically turbulent, and since NRT is stronger than SRT, we see that $E_G^{{{\rm{EES}}(H)}}$ is Borel reducible (in fact continuously embeddable) to $E_G^{{\rm{SA}}(H)}$ by the inclusion map $\iota\colon ({\rm{EES}}(H),{\rm{NRT}})\to ({\rm{SA}}(H),{\rm{SRT}})$.
\end{proof}
We now show that the action $\beta$ is weakly generically turbulent and use Theorem \ref{thm: weak turbulence=turbulence}. 
\begin{proposition}\label{prop: yasu G-orbit is dense meager}
For any $A\in {\rm{EES}}(H)$, the orbit $[A]_G$ is ${\rm{NRT}}$-dense and meager in ${\rm{EES}}(H)$. 
\end{proposition}
For the proof, we use an easy lemma.
\begin{lemma}\label{lem: spectrum and perturbation by K}
Let $A\in {\rm{SA}}(H)$, $\lambda\in \sigma(A)$, $K\in \mathbb{K}(H)_{\rm{sa}}$ and $c>1$. Then $\sigma(A+K)\cap [\lambda-c\|K\|,\lambda+c\|K\|]\neq \emptyset$. 
\end{lemma}
\begin{proof}
Suppose by contradiction that $B:=A+K$ satisfies $\sigma(B)\cap [\lambda-c\|K\|,\lambda+c\|K\|]= \emptyset$. Then for $\mu \in \sigma(B)$, $|\mu-\lambda|\ge c\|K\|$, and hence $\|(B-\lambda)^{-1}\|\le (c\|K\|)^{-1}$. It follows that
\[A-\lambda=B-K-\lambda=(B-\lambda)(1-(B-\lambda)^{-1}K).\]
Since $\|(B-\lambda)^{-1}K\|\le c^{-1}<1$, $1-(B-\lambda)^{-1}K$ is invertible with bounded inverse, whence $A-\lambda$ also has the bounded inverse $(A-\lambda)^{-1}=(1-(B-\lambda)^{-1}K)^{-1}(B-\lambda)^{-1}$, which contradicts $\lambda \in \sigma(A)$. 
\end{proof}

\begin{proof}[Proof of Proposition \ref{prop: yasu G-orbit is dense meager}]
First we show that the orbit $[A]_G$ is dense in ${\rm{EES}}(H)$. Let $B\in {\rm{EES}}(H)$. Then there exists CONS $\{\xi_n\}_{n=1}^{\infty}$ (resp. $\{\eta_n\}_{n=1}^{\infty}$) for $H$ and a real sequence $\{a_n\}_{n=1}^{\infty}$ (resp. $\{b_n\}_{n=1}^{\infty}$) such that $A=\sum_{n=1}^{\infty}a_n\nai{\xi_n}{\ \cdot\ }\xi_n$ and $B=\sum_{n=1}^{\infty}b_n\nai{\eta_n}{\ \cdot\ }\eta_n$. Find $u\in \mathcal{U}(H)$ such that $u\xi_n=\eta_n\ (n\in \mathbb{N})$. Put $K_N:=\sum_{n=1}^N(b_n-a_n)\nai{\eta_n}{\ \cdot\ }\eta_n\ (N\in \mathbb{N})$. Then 
\[uAu^*+K_N=\sum_{n=1}^Nb_n\nai{\eta_n}{\ \cdot\ }\eta_n+\sum_{n=N+1}^{\infty}a_n\nai{\eta_n}{\ \cdot\ }\eta_n.\]
Since $A,B\in {\rm{EES}}(H)$, $|a_n|,|b_n|\to \infty$ as $n\to \infty$. 
Therefore 
\eqa{
\|(uAu^*+K_N-i)^{-1}-(B-i)^{-1}\|&=\left \|\sum_{n=N+1}^{\infty}\left (\frac{1}{a_n-i}-\frac{1}{b_n-i}\right )\nai{\eta_n}{\ \cdot\ }\eta_n\right \|\\
&=\sup_{n\ge N+1}\left |\frac{1}{a_n-i}-\frac{1}{b_n-i}\right |\\
&\stackrel{N\to \infty}{\rightarrow }0.
}
Therefore $B$ is in the NRT-closure of $[A]_G$. Thus every orbit is dense.

Next we show that $[A]_G$ is meager. 
Let $0\neq K\in \mathbb{K}(H)_{\rm{sa}}$. Then choose a constant $c>1$ such that $q:=c\|K\|\in \mathbb{Q}_{>0}$. By Lemma \ref{lem: spectrum and perturbation by K}, we have $\sigma(A+K)\cap [\lambda-q,\lambda+q]\neq \emptyset$ for each $\lambda\in \sigma(A)=\sigma_{\rm{p}}(A)$. Thus we have (note that since $H$ is separable, $\sigma_{\rm{p}}(A)$ is at most countable)
\begin{equation}
[A]_G\subset \bigcup_{q\in \mathbb{Q}_{>0}}\bigcap_{\lambda \in \sigma_{\rm{p}}(A)} \underbrace{\left \{B\in {\rm{EES}}(H);\ \sigma_{\rm{p}}(B)\cap [\lambda-q,\lambda+q]\neq \emptyset \right \}}_{=:S_{q,\lambda}}.\label{eq: S_q,lambda}
\end{equation}

We show that the right hand side of (\ref{eq: S_q,lambda}) is meager. This is done in two steps:\\ \\
\textbf{Step 1.} $S_{q,\lambda}$ is NRT-closed for each $q\in \mathbb{Q}_{>0}, \lambda\in \sigma_{\rm{p}}(A)$.\\
Let $S_{q,\lambda}\ni B_n\stackrel{n\to \infty}{\to}B\in {\rm{EES}}(H)$ (NRT). Assume that $\sigma_{\rm{p}}(B)\cap [\lambda-q,\lambda+q]=\emptyset$. 
Therefore $\lambda\pm q\notin \sigma (B)$. Since $\mathbb{C}\setminus \sigma(B)$ is open, there exists $\varepsilon>0$ such that $[\lambda-q-\varepsilon,\lambda+q+\varepsilon]\cap \sigma(B)=\emptyset$. By \cite[Theorem VIII.23]{ReedSimonI}, $P_n:=E_{B_n}((\lambda-q-\frac{\varepsilon}{2},\lambda+q+\frac{\varepsilon}{2}))$ converges to $E_B((\lambda-q-\frac{\varepsilon}{2},\lambda+q+\frac{\varepsilon}{2})=0$ in norm. Since $P_n\ (n\in \mathbb{N})$ is a projection, this shows that there exists $n_0\ge 1$ such that $P_n=0\ (n\ge n_0)$. 
This means in particular that $\sigma_{\rm{p}}(B_{n_0})\cap [\lambda-q,\lambda+q]=\emptyset$, a contradiction. Therefore $S_{q,\lambda}$ is NRT-closed. \\ \\
\textbf{Step 2.} $\mathcal{S}_q:=\bigcap_{\lambda \in \sigma_{\rm{p}}(A)}S_{q,\lambda}$ is a (closed) nowhere-dense subset of ${\rm{EES}}(H)$.\\
Assume by contradiction that there exists $B\in \mathcal{S}_q$ and $\varepsilon>0$ such that $\mathcal{S}_q$ contains an open neighborhood $\{C\in {\rm{EES}}(H);\ \|(B-i)^{-1}-(C-i)^{-1}\|<\varepsilon\}$ of $B$. 
Let 
\[
A=\sum_{n=1}^{\infty}a_n\nai{\xi_n}{\ \cdot\ }\xi_n,\ \ \ \ 
B=\sum_{n=1}^{\infty}b_n\nai{\eta_n}{\ \cdot\ }\eta_n,
\]
where $\{\xi_n\}_{n=1}^{\infty},\ \{\eta_n\}_{n=1}^{\infty}$ are CONSs for $H$, and $|a_n|,|b_n|\nearrow \infty$.
Since $|b_n|\nearrow \infty$, there is $n_0\in \mathbb{N}$ such that $(|b_n|^2+1)^{-1/2}<\varepsilon/2$ for $n>n_0$.  Since $|a_n|\nearrow \infty$, there is $n_1\in \mathbb{N}$ such that $|a_{n_1}|>q$ and $|b_{n_0}|<|a_{n_1}|-q$ holds. Then we may also find $n_2\in \mathbb{N}$ such that 
$|a_{n_1}|+q<|b_{n_2}|$ and $n_2>n_0$ hold. Now define $C\in {\rm{EES}}(H)$ by

\[C:=\sum_{n=1}^{\infty}c_n\nai{\eta_n}{\ \cdot\ }\eta_n,\ \ c_n:=\begin{cases}
\ \ \ \ \ b_n & (1\le n\le n_0)\\
b_{n_2+(n-n_0)} & (n>n_0),
\end{cases}.\]
By construction, we have 
\begin{equation}
|c_n|\le |b_{n_0}|<|a_{n_1}|-q\ (n\le n_0),\ \ \ \  |c_n|\ge |b_{n_2}|>|a_{n_1}|+q\ (n>n_0).\label{eq: c_n has a gap}
\end{equation}
We compute 
\eqa{
\|(C-i)^{-1}-(B-i)^{-1}\|&\le \left \|\sum_{n=n_0+1}^{\infty}\left (\frac{1}{b_n-i}-
\frac{1}{b_{n_2+(n-n_0)}-i}\right )\nai{\eta_n}{\ \cdot\ }\eta_n \right \|\\
&\le \sup_{n\ge n_0+1}\left (\frac{1}{\sqrt{|b_n|^2+1}}+\frac{1}{\sqrt{|b_{n_2+(n-n_0)}|^2+1}}\right )\\
&<\frac{\varepsilon}{2}+\frac{\varepsilon}{2}=\varepsilon,
}
which shows by our assumption that $C\in \mathcal{S}_{q}$. However, we have $\sigma(C)\cap [a_{n_1}-q,a_{n_1}+q]=\emptyset$ by (\ref{eq: c_n has a gap}), which is a contradiction. Therefore $\mathcal{S}_{q}$ is nowhere-dense. By Step 1 and Step 2, we have shown that $[A]_G$ is meager.  
\end{proof}

Finally, we show that the action of $G$ on ${\rm{EES}}(H)$ satisfies condition (b) of Definition \ref{def: weak turbulence}.
We need the following elementary but useful estimate.
\begin{lemma}\label{lem: nonlinearity of resolvent does not appear if ab>-1}
Let $a,b\in \mathbb{R}$ and let $0\le s\le 1$. 
If $ab\ge -1$, then 
\[\left |\frac{1}{(1-s)a+sb-i}-\frac{1}{a-i}\right |\le \left |\frac{1}{b-i}-\frac{1}{a-i}\right |.\]
\end{lemma}
\begin{remark}
If $ab< -1$, then the left hand side in the above, regarded as a function in $s\in [0,1]$ attains the maximum value 1 at $s=\frac{1+a^2}{(a^2-ab)}$. 
\end{remark}
\begin{proof}[Proof of Lemma \ref{lem: nonlinearity of resolvent does not appear if ab>-1}]
We may and do assume that $a\neq b$. Define $f\colon [0,1]\to \mathbb{R}$ by
\eqa{f(s)&=\left |\frac{1}{(1-s)a+sb-i}-\frac{1}{a-i}\right |^2\\
&=\frac{s^2(a-b)^2}{[\{(1-s)a+sb\}^2+1](a^2+1)}\\
&=:\frac{(a-b)^2}{a^2+1}g(s).
}
Since $0\le s\le 1$ and $f(0)=0$, we consider the case $0<s\le 1$. We compute
\eqa{
g(s)
&=\frac{1}{(a-b)^2-2(a-b)as^{-1}+(1+a^2)s^{-2}}\\
&\stackrel{s^{-1}=t}{=}\frac{1}{(a-b)^2-2(a-b)at+(1+a^2)t^2}\\
&=:\frac{1}{h(t)}.
}
Therefore we consider the minimum of $h(t)$. Since $0<s\le 1$, $t\ge 1$. We see that
\eqa{
h(t)&=(1+a^2)\left \{t^2-\frac{2(a-b)a}{a^2+1}t+\frac{(a-b)^2}{a^2+1}\right \}\\
&=(1+a^2)\left \{\left (t-\frac{(a-b)a}{a^2+1}\right )^2-\frac{a^2(a-b)^2}{(a^2+1)^2}+\frac{(a-b)^2}{a^2+1}\right \}.
}
Since $ab\ge -1\Leftrightarrow t_0:=\frac{(a-b)a}{a^2+1}\le 1$, $h\colon [1,\infty)\to \mathbb{R}$ takes the minimum value at $t=1$, whence $f(s)\le f(1)\ (0\le s\le 1)$ holds. 
\end{proof}

\begin{lemma}\label{lem: unbounded below/above comeager}
${\rm{EES}}_{\pm \infty}(H):=\{A\in {\rm{EES}}(H);\ \inf \sigma(A)=-\infty,\sup \sigma(A)=\infty\}$ is a $G$-invariant dense $G_{\delta}$ subset of ${\rm{EES}}(H)$.
\end{lemma}
\begin{proof}
It is easy to see the $G$-invariance of ${\rm{EES}}_{\pm \infty}(H)$. 
Define subsets of ${\rm{EES}}(H)$ by
\eqa{
{\rm{EES}}_{>-\infty}(H)=\{A;\ \inf \sigma(A)>-\infty\},\ \ \ \ &{\rm{EES}}_{\ge n}(H)=\{A;\ \inf \sigma(A)\ge n\}\\
{\rm{EES}}_{<\infty}(H)=\{A;\ \sup \sigma(A)<\infty\},\ \ \ \ &{\rm{EES}}_{\le n}(H)=\{A;\ \sup \sigma(A)\le n\}
}
Then 
\eqa{
{\rm{EES}}_{\pm \infty}(H)&={\rm{EES}}(H)\setminus ({\rm{EES}}_{>-\infty}(H)\cup {\rm{EES}}_{<\infty}(H)),\\
{\rm{EES}}_{>-\infty}(H)&=\bigcup_{n\in \mathbb{Z}}{\rm{EES}}_{\ge n}(H),\ \ \ \ {\rm{EES}}_{<\infty}(H)=\bigcup_{n\in \mathbb{Z}}{\rm{EES}}_{\le n}(H).
}
Therefore it suffices to show that ${\rm{EES}}_{\ge n}(H),\ {\rm{EES}}_{\le n}(H)$ are closed, nowhere-dense subsets for every $n\in \mathbb{Z}$. Let $n\in \mathbb{Z}$, and we first prove that  ${\rm{EES}}_{\ge n}(H)$ is closed: note that 
\eqa{
{\rm{EES}}_{\ge n}(H)&=\{A\in {\rm{EES}}(H); \sigma(A)\cap (-\infty,n)=\emptyset\}\\
&=\bigcap_{m=1}^{\infty}\underbrace{\{A\in {\rm{EES}}(H); \sigma(A)\cap (n-m,n)=\emptyset\}}_{=:F_{n,m}}.
}
Hence it suffices to show that each $F_{n,m}\ (m\in \mathbb{N})$ is closed in ${\rm{EES}}(H)$. But this follows from \cite[Theorem VIII.24 (a)]{ReedSimonI} because SRT is weaker than NRT. 

Next, we show that ${\rm{EES}}_{\ge n}(H)$ is nowhere-dense: let $A\in {\rm{EES}}_{\ge n}(H)$.  Then we have the spectral resolution $A=\sum_{m=1}^{\infty}a_m\nai{\xi_m}{\ \cdot\ }\xi_m$ with $a_m\ge n\ (m\in \mathbb{N})$ and $\{\xi_m\}_{m=1}^{\infty}$ a CONS for $H$. Then for each $k\in \mathbb{N}$, define
\[A_k:=\sum_{m\neq k}a_m\nai{\xi_m}{\ \cdot\ }\xi_m+(n-k)\nai{\xi_k}{\ \cdot\ }\xi_k\in {\rm{EES}}(H).\]
Then $A_k\notin {\rm{EES}}_{\ge n}(H)$, but
\[\|(A-i)^{-1}-(A_k-i)^{-1}\|=\left |\frac{1}{a_k-i}-\frac{1}{n-k-i}\right |\stackrel{k\to \infty}{\to}0,\]
so $A\notin {\rm{Int}}({\rm{EES}}_{\ge n}(H))$. Therefore ${\rm{Int}}({\rm{EES}}_{\ge n}(H))=\emptyset$. Similarly, it can be shown that ${\rm{EES}}_{\le n}(H)$ is closed and nowhere-dense for each $n\in \mathbb{Z}$. Therefore both ${\rm{EES}}_{>-\infty}(H)$ and ${\rm{EES}}_{<\infty}(H)$ are meager $F_{\sigma}$ sets, and the proof is completed. 
\end{proof}
\begin{proof}[Proof of Theorem \ref{thm: action is weakly turbulent}]
By Theorem \ref{thm: weak turbulence=turbulence}, it is enough to show that the action is weakly generically turbulent. We have shown that all orbits are dense and meager (Proposition \ref{prop: yasu G-orbit is dense meager}). Therefore we have only to prove (b) in Definition \ref{def: weak turbulence}.  Let $A,B\in {\rm{EES}}_{\pm \infty}(H)$, and let $U$ be an open neighborhood of $A$ in ${\rm{EES}}(H)$, $V$ be an open neighborhood of 1 in $G$. We may and do assume that $U,V$ are of the form $U=\{C\in {\rm{EES}}(H);\ \|(A-i)^{-1}-(C-i)^{-1}\|<\delta\}$, and $V=W_1\times W_2$, where $W_1=\{K\in \mathbb{K}(H)_{\rm{sa}};\ \|K\|<r\}$ and $W_2$ is an open neighborhood of 1 in $\mathcal{U}(H)$. We prove that 
$\overline{\mathcal{O}(A,U,V)}\cap [B]_G\neq \emptyset$, which shows (b) because by Lemma \ref{lem: unbounded below/above comeager}, ${\rm{EES}}_{\pm \infty}(H)$ is comeager in ${\rm{EES}}(H)$. 
Let $A=\sum_{n=1}^{\infty}a_n\nai{\xi_n}{\ \cdot\ }\xi_n,\ B=\sum_{n=1}^{\infty}b_n\nai{\eta_n}{\ \cdot\ }\eta_n$ be the spectral resolutions of $A,B$ respectively. 
Define $v\in \mathcal{U}(H)$ by $v\eta_n:=\xi_n\ (n\in \mathbb{N})$. Then 
\[B_1:=vBv^*=\sum_{n=1}^{\infty}b_n\nai{\xi_n}{\ \cdot\ }\xi_n\in [B]_G.\]
Let $I_A:=\{n\in \mathbb{N};a_n\ge 0\},\ J_A:=\{n\in \mathbb{N}; a_n<0\}$ and define $I_B,J_B\subset \mathbb{N}$ analogously. By assumption, all $I_A,J_A,I_B,J_B$ are infinite, so write
\eqa{
I_A=\{n_1<n_2<\cdots\},\ \ \ \ &J_A=\{n_1'<n_2'<\cdots\}\\
I_B=\{m_1<m_2<\cdots,\},\ \ \ \ &J_B=\{m_1'<m_2'<\cdots \}.
}
Define a permutation $\pi$ of $\mathbb{N}$ by $\pi(n_k):=m_k,\ \pi(n_k')=m_k'$, and define $u_{\pi}\in \mathcal{U}(H)$ by $u_{\pi}\xi_n:=\xi_{\pi^{-1}(n)}\ (n\in \mathbb{N})$. Then for each $k\in \mathbb{N}$, $u_{\pi}B_1u_{\pi}^*\xi_{n_k}=b_{m_k}\xi_{n_k},\ u_{\pi}B_1u_{\pi}^*\xi_{n_k'}=b_{m_k'}\xi_{n_k'}$, 
and
\[B_2:=u_{\pi}B_1u_{\pi}^*=\sum_{k=1}^{\infty}b_{m_k}\nai{\xi_{n_k}}{\ \cdot\ }\xi_{n_k}+\sum_{k=1}^{\infty}b_{m_k'}\nai{\xi_{n_k'}}{\ \cdot\ }\xi_{n_k'}\in [B]_G.\]
Then by the choice of $I_A,J_A,I_B,J_B$, we now have $a_{n_k}b_{m_k}\ge 0,\ a_{n_k'}b_{m_k'}\ge 0$, 
so that if we write $B_2=\sum_{n=1}^{\infty}\tilde{b}_n\nai{\xi_n}{\ \cdot\ }\xi_n$, we have
\begin{equation}
a_n\tilde{b}_n\ge 0\ \ \ \ (n\in \mathbb{N}). \label{eq: product is always positive}
\end{equation}
Next, let $K_N:=\sum_{n=1}^N(-\tilde{b}_n+a_n)\nai{\xi_n}{\ \cdot\ }\xi_n\in \mathbb{K}(H)_{\rm{sa}}$, and consider $C_N:=B_2+K_N\in [B]_G$. Then as $|\tilde{b}_n|, |a_n|\to \infty\ (n\to \infty)$, we have
\[\|(C_N-i)^{-1}-(A-i)^{-1}\|=\sup_{n\ge N+1}\left |\frac{1}{\tilde{b}_n-i}-\frac{1}{a_n-i}\right |\\
\stackrel{N\to \infty}{\to}0,\]
so that there exists $N\in \mathbb{N}$ for which $\|(C_N-i)^{-1}-(A-i)^{-1}\|<\delta.\label{eq: C_N is in U}$ holds. In particular, $C_N\in U$.\\ \\
\textbf{Claim.} $C_N\in \overline{\mathcal{O}(A,U,V)}\cap [B]_G$.\\
The proof of the claim would conclude that (b) holds. To show that $C_N\in \overline{\mathcal{O}(A,U,V)}$, define for each $p\ge N+1$ an operator 
\[C_{N,p}:=\sum_{n=1}^Na_n\nai{\xi_n}{\ \cdot\ }\xi_n+\sum_{n=N+1}^p\tilde{b}_n\nai{\xi_n}{\ \cdot\ }\xi_n+\sum_{n=p+1}^{\infty}a_n\nai{\xi_n}{\ \cdot\ }\xi_n.\]
Then we have
\eqa{
\|(C_{N,p}-i)^{-1}-(A-i)^{-1}\|&=\sup_{N+1\le n\le p}\left |\frac{1}{\tilde{b}_n-i}-\frac{1}{a_n-i}\right |\\ &\le \|(C_N-i)^{-1}-(A-i)^{-1}\|\\
&<\delta,
}
so $C_{N,p}\in U\ (p\ge N+1)$ holds. 
Moreover, we see that
\[
\|(C_N-i)^{-1}-(C_{N,p}-i)^{-1}\|=\sup_{n\ge p+1}\left |\frac{1}{\tilde{b}_n-i}-\frac{1}{a_n-i}\right |
\stackrel{p\to \infty}{\to}0.\]
We now show that $C_{N,p}\in \mathcal{O}(A,U,V)$, which implies $C_N\in \overline{\mathcal{O}(A,U,V)}$.  
Put $m_p:=\max_{N+1\le n\le p}|\tilde{b}_n-a_n|$, and choose $L\in \mathbb{N}$ so that $m_p<rL$. Define
\[K:=\sum_{n=N+1}^p\frac{\tilde{b}_n-a_n}{L}\nai{\xi_n}{\ \cdot\ }\xi_n\in \mathbb{K}(H)_{\rm{sa}}.\]
Then $\|K\|=\frac{m_p}{L}<r$, whence $K\in W_1$. Therefore $g=(K,1)\in V$. For each $0\le j\le L$, define 
$A_j:=A+jK=g^j\cdot A$. In particular, $A_0=A,A_L=C_{N,p}$. Now by (\ref{eq: product is always positive}) and Lemma \ref{lem: nonlinearity of resolvent does not appear if ab>-1}, we have
\eqa{
\|(A_j-i)^{-1}-(A-i)^{-1}\|&=\left \|\sum_{n=N+1}^p\left (\frac{1}{a_n+\frac{j}{L}(\tilde{b}_n-a_n)-i}-\frac{1}{a_n-i}\right )\nai{\xi_n}{\ \cdot\ }\xi_n\right \|\\
&=\sup_{N+1\le n\le p}\left |\frac{1}{a_n+\frac{j}{L}(\tilde{b}_n-a_n)-i}-\frac{1}{a_n-i}\right |\\
&\le \sup_{N+1\le n\le p}\left |\frac{1}{\tilde{b}_n-i}-\frac{1}{a_n-i}\right |\\
&=\|(C_{N,p}-i)^{-1}-(A-i)^{-1}\|\\
&<\delta.
}
Therefore $A_j\in U$ for each $0\le j\le L$, whence $C_{N,p}\in \mathcal{O}(A,U,V)$ and the claim is proved. This shows that the action is weakly generically turbulent, so it is generically turbulent.  
\end{proof}
\section{More Borel Equivalence Relations}
In this section, we consider several Borel equivalence relations related to the Weyl-von Neumann equivalence relation. 
\subsection{$E_{\rm{u.c.res}}^{{\rm{SA}}(H)}$ is Smooth}
\begin{definition}
We define an equivalence relation $E_{\rm{u.c.res}}^{{\rm{SA}}(H)}$ (``unitary equivalence modulo compact difference of resolvents") on ${\rm{SA}}(H)$ by 
$AE_{\rm{u.c.res}}^{{\rm{SA}}(H)}B$ if and only if there exists $u\in \mathcal{U}(H)$ such that $u(A-i)^{-1}u^*-(B-i)^{-1}\in \mathbb{K}(H)$. 
\end{definition}
It is easy to see that $E_{\rm{u.c.res}}^{{\rm{SA}}(H)}$ is an equivalence relation. Note that Weyl-von Neumann equivalence relation $E_G^{{\rm{SA}}(H)}$ is ``stronger" than $E_{\rm{u.c.res}}^{{\rm{SA}}(H)}$: 
\begin{lemma}\label{lem: von Neuman is stronger than E_res.c}
$E_G^{{\rm{SA}}(H)}\subset E_{\rm{u.c.res}}^{{\rm{SA}}(H)}$ holds. 
\end{lemma}
\begin{proof}
Let $A,B\in {\rm{SA}}(H)$ be such that $AE_G^{{\rm{SA}}(H)}B$. Then there exist $u\in \mathcal{U}(H)$ and $K\in \mathbb{K}(H)_{\rm{sa}}$ such that $B=uAu^*+K$. Then by the resolvent identity \cite[$\S$2.2, (2.4)]{Schmudgen}
\eqa{
(B-i)^{-1}-u(A-i)^{-1}u^*&=(B-i)^{-1}-(uAu^*-i)^{-1}\\
&=(B-i)^{-1}(uAu^*-B)(uAu^*-i)^{-1}\\
&=-(B-i)^{-1}K(uAu^*-i)^{-1}\in \mathbb{K}(H),
}
whence $AE_{\rm{u.c.res}}^{{\rm{SA}}(H)}B$. 
\end{proof}
It turns out that the restriction of $E_{\rm{u.c.res}}^{{\rm{SA}}(H)}$ to the $F_{\sigma}$ subset $\mathbb{B}(H)_{\rm{sa}}$ coincides with $E_G^{\mathbb{B}(H)_{\rm{sa}}}$. 
\begin{lemma}\label{lem: Matsuzawa relation coincides with von Neumann's on B(H)}
The restriction of $E_{\rm{u.c.res}}^{{\rm{SA}}(H)}$ to $\mathbb{B}(H)_{\rm{sa}}$ coincides with $E_G^{\mathbb{B}(H)_{\rm{sa}}}$.
\end{lemma}
\begin{proof}
Let $A,B\in \mathbb{B}(H)_{\rm{sa}}$. If $AE_G^{\mathbb{B}(H)_{\rm{sa}}}B$, then $AE_{\rm{u.c.res}}^{{\rm{SA}}(H)}B$ by Lemma \ref{lem: von Neuman is stronger than E_res.c}. Conversely, assume that $AE_{\rm{u.c.res}}^{{\rm{SA}}(H)}B$ holds. Then there exists $u\in \mathcal{U}(H)$ such that 
\[
(B-i)^{-1}-(uAu^*-i)^{-1}=(B-i)^{-1}(uAu^*-B)(uAu^*-i)^{-1}\in \mathbb{K}(H).
\]
Let $K:=(B-i)^{-1}-(uAu^*-i)^{-1}$. Then because $A,B$ are bounded and self-adjoint, we have
\[B-uAu^*=-(B-i)K(uAu^*-i)\in \mathbb{K}(H)_{\rm{sa}}.\]
This shows that $AE_G^{\mathbb{B}(H)_{\rm{sa}}}B$. 
\end{proof}
Therefore $E_{\rm{u.c.res}}^{{\rm{SA}}(H)}$ is considered to be another generalization of the smooth equivalence relation  $E_G^{\mathbb{B}(H)_{\rm{sa}}}$ to general self-adjoint operators. We have seen that $E_G^{{\rm{SA}}(H)}$ is unclassifiable by countable structure. However, it turns out that apparently similar equivalence relation  $E_{\rm{u.c.res}}^{{\rm{SA}}(H)}$ is actually smooth:  
\begin{theorem}\label{thm: Matsuzawa relation is smooth}
$E_{\rm{u.c.res}}^{{\rm{SA}}(H)}$ is a smooth equivalence relation.
\end{theorem}
Before going to the proof, note that the essential spectra is not a complete invariant for $E_{\rm{u.c.res}}^{{\rm{SA}}(H)}$:
\begin{example}\label{ex: bounded vs unbounded}
Consider $H=H_0\oplus H_0$ where $H_0$ is a separable infinite-dimensional Hilbert space, and let $A_0\in {\rm{EES}}(H_0)$. Then $A:=A_0\oplus 0,B:=0\oplus 0\in {\rm{SA}}(H)$ satisfy $\sigma_{\rm{ess}}(A)=\sigma_{\rm{ess}}(B)=\{0\}$, but for any $u\in \mathcal{U}(H)$, 
\[
(A-i)^{-1}-u(B-i)^{-1}u^*=[(A_0-i1_{H_0})^{-1}-i1_{H_0}]\oplus 0\notin \mathbb{K}(H),
\]
because $(A_0-i1_{H_0})^{-1}\in \mathbb{K}(H)_{\rm{sa}}$ (Lemma \ref{lem: ess=empty iff compact resolvent}) and $1_{H_0}\notin \mathbb{K}(H_0)$. Therefore $(A,B)\notin  E_{\rm{u.c.res}}^{{\rm{SA}}(H)}$. 
\end{example}
Note that in Example \ref{ex: bounded vs unbounded}, $A$ is unbounded, while $B$ is bounded. It turns out that if we add to $\sigma_{\rm{ess}}(\cdot)$ the additional information of boundedness/unboundedness of the operator, then it becomes a complete invariant for $E_{\rm{u.c.res}}^{{\rm{SA}}(H)}$.
\begin{definition}
For each $A\in {\rm{SA}}(H)$, we define $\overline{\sigma}_{\rm{ess}}(A)\in \mathcal{F}(\mathbb{R})\times \{0,1\}$ by
\[\overline{\sigma}_{\rm{ess}}(A):=\begin{cases}\ (\sigma_{\rm{ess}}(A),0)& (A\text{ is bounded\ })\\
\ (\sigma_{\rm{ess}}(A),1)& (A\text{ is unbounded\ })\end{cases}.\]
\end{definition}
$\overline{\sigma}_{\rm{ess}}(A)$ is something like a compactification of $\sigma_{\rm{ess}}(A)$. Note that since $\mathbb{B}(H)_{\rm{sa}}$ is a Borel subset of ${\rm{SA}}(H)$ (Lemma \ref{lem: B(H)_sa is standard}), the map $\overline{\sigma}_{\rm{ess}}\colon {\rm{SA}}(H)\to \mathcal{F}(\mathbb{R})\times \{0,1\}$ is Borel by the Borelness of $\sigma_{\rm{ess}}(\cdot)$ (Theorem \ref{thm: taking ess spectrum is Borel}). Now Theorem \ref{thm: Matsuzawa relation is smooth} is proved by the next Proposition: 
\begin{proposition}\label{prop: Matsuzawa relation is smooth}
$\overline{\sigma}_{\rm{ess}}$ is a Borel reduction of $E_{\rm{u.c.res}}^{{\rm{SA}}(H)}$ to ${\rm{id}}_{\mathcal{F}(\mathbb{R})\times \{0,1\}}$. In particular, $E_{\rm{u.c.res}}^{{\rm{SA}}(H)}$ is smooth. 
\end{proposition}
We need preparations. We use the following well-known variant of Weyl's criterion (Theorem \ref{thm: Weyl criterion}):
\begin{lemma}\label{lem: Weyl criterion modified}
Let $A\in {\rm{SA}}(H)$ and $\lambda \in \mathbb{R}$. The following conditions are equivalent: 
\begin{list}{}{}
\item[{\rm{(i)}}] $\lambda \in \sigma_{\rm{ess}}(A)$.
\item[{\rm{(ii)}}] There exists a sequence $\{\xi_n\}_{n=1}^{\infty}$ of unit vectors in $H$ which converges weakly to 0, such that $\|(A-i)^{-1}\xi_n-(\lambda-i)^{-1}\xi_n\|\stackrel{n\to \infty}{\to}0$.
\end{list}
\end{lemma}
\begin{proof}
(i)$\Rightarrow$(ii): By Theorem \ref{thm: Weyl criterion}, there exists 
a sequence $\{\xi_n\}_{n=1}^{\infty}$ of unit vectors in $\dom{A}$ converging weakly to 0, such that $\|(A-\lambda)\xi_n\|\stackrel{n\to \infty}{\to}0$. Then we have 
\eqa{
\|(A-i)^{-1}\xi_n-(\lambda-i)^{-1}\xi_n\|^2&=\int_{\mathbb{R}}\left |\frac{1}{\mu-i}-\frac{1}{\lambda-i}\right |^2d\|E_A(\mu)\xi_n\|^2\\
&=\int_{\mathbb{R}}\frac{(\lambda-\mu)^2}{(\mu^2+1)(\lambda^2+1)}d\|E_A(\mu)\xi_n\|^2\\
&\le \int_{\mathbb{R}}(\lambda-\mu)^2d\|E_A(\mu)\xi_n\|^2=\|(A-\lambda)\xi_n\|^2\\
&\stackrel{n\to \infty}{\to}0.
}
(ii)$\Rightarrow $(i) Assume by contradiction that $\lambda\notin \sigma_{\rm{ess}}(A)$. Let $\{\xi_n\}_{n=1}^{\infty}$ be a sequence of unit vectors in $H$ satisfying (ii). By assumption, there exists $\varepsilon>0$ such that $\text{rank} E_A(B_{\varepsilon}(\lambda))<\infty$, where $B_{\varepsilon}(\lambda)=(\lambda-\varepsilon,\lambda+\varepsilon)$. Then $P:=E_A(B_{\varepsilon}(\lambda))$ is compact, whence $\|P\xi_n\|\stackrel{n\to \infty}{\to}0$. It then follows that 
\eqa{
\|\{(A-i)^{-1}-(\lambda-i)^{-1}\}\xi_n\|^2&=\int_{\mathbb{R}}\frac{(\lambda-\mu)^2}{(\mu^2+1)(\lambda^2+1)}d\|E_A(\mu)\xi_n\|^2\\
&\ge \int_{|\mu-\lambda|\ge \varepsilon}\frac{(\lambda-\mu)^2}{(\mu^2+1)(\lambda^2+1)}d\|E_A(\mu)\xi_n\|^2\\
&\ge \frac{\varepsilon^2}{\lambda^2+1}\int_{|\mu-\lambda|\ge \varepsilon}\frac{1}{\mu^2+1}d\|E_A(\mu)\xi_n\|^2\\
&=\frac{\varepsilon^2}{\lambda^2+1}\|(A-i)^{-1}(1-P)\xi_n\|^2.
}
Since $\|P\xi_n\|\to 0$ and $(A-i)^{-1}$ is bounded, by condition (ii) it holds that 
$\|(A-i)^{-1}\xi_n\|\stackrel{n\to \infty}{\to}0$. However, this implies by $\|\{(A-i)^{-1}-(\lambda-i)^{-1}\}\xi_n\|\stackrel{n\to \infty}{\to}0$ and $\|\xi_n\|=1\ (n\in \mathbb{N})$ that 
\[0\neq |(\lambda-i)^{-1}|=\|(\lambda-i)^{-1}\xi_n\|\stackrel{n\to \infty}{\to}0,\]
which is a contradiction. 
\end{proof}
We also use Weyl's criterion for bounded normal operators. Recall that the essential spectra $\sigma_{\rm{ess}}(x)$ for a bounded normal operator $x\in \mathbb{B}(H)$ is defined in the same way as the case of self-adjoint operators: $\sigma_{\rm{ess}}(x)=\sigma(x)\setminus \sigma_{\rm{d}}(x)$, where $\sigma_{\rm{d}}(x)$ is the set of all eigenvalues of $x$ of finite multiplicity. The next lemma can be proved by the same argument as Weyl's criterion (Theorem \ref{thm: Weyl criterion}): 
\begin{lemma}[Weyl's criterion for normal operators]\label{lem: Weyl criterion for normal operators}
Let $x\in \mathbb{B}(H)$ be a normal operator, and let $\lambda \in \mathbb{C}$. Then $\lambda \in \sigma_{\rm{ess}}(x)$ if and only if there exists a sequence $\{\xi_n\}_{n=1}^{\infty}$ of unit vectors in $H$ converging weakly to 0, such that $\|x\xi_n-\lambda\xi_n\|\stackrel{n\to \infty}{\to}0$. 
\end{lemma}
\begin{corollary}\label{cor: ess spec of resolvent}
Let $A\in {\rm{SA}}(H)$. Then $A$ is bounded if and only if $0\notin \sigma_{\rm{ess}}((A-i)^{-1})$. 
Moreover, it holds that 
\[\sigma_{\rm{ess}}((A-i)^{-1})=\begin{cases}
\{(\lambda-i)^{-1};\ \lambda\in \sigma_{\rm{ess}}(A)\}& (A{\rm{\ is\ bounded}})\\
\{(\lambda-i)^{-1};\ \lambda\in \sigma_{\rm{ess}}(A)\}\cup \{0\} & (A{\rm{\ is\ ubounded}})
\end{cases}
\]
\end{corollary}
\begin{proof}
Assume that $A$ is bounded. Then we have $\|(A-i)^{-1}\xi \|^2\ge \frac{1}{\|A\|^2+1}\|\xi\|^2\ (\xi \in H)$ (see (\ref{eq: bounded operator and resolvent}) in Lemma \ref{lem: B(H)_sa is standard}). In particular, $0\notin \sigma((A-i)^{-1})$, whence $0\notin \sigma_{\rm{ess}}((A-i)^{-1})$. If $A$ is unbounded, then By Lemma \ref{lem: Weyl criterion modified}, $0\in \sigma_{\rm{ess}}((A-i)^{-1})$ holds. Therefore the first claim follows. The rest of the statements are immediate corollaries of Lemma \ref{lem: Weyl criterion modified} in combination with Weyl's criterion (Lemma \ref{lem: Weyl criterion for normal operators}). 
\end{proof}
\begin{lemma}\label{lem: A,B equiv then sigma bar are equal}
Let $A,B\in {\rm{SA}}(H)$ be such that $AE_{\rm{u.c.res}}^{{\rm{SA}}(H)}B$. Then $\overline{\sigma}_{\rm{ess}}(A)=\overline{\sigma}_{\rm{ess}}(B)$. 
\end{lemma}
\begin{proof}
By assumption, there exists $u\in \mathcal{U}(H)$ such that $u(A-i)^{-1}u^*-(B-i)^{-1}\in \mathbb{K}(H)$. 
Since the essential spectra of a bounded normal operators is invariant under compact perturbations (see e.g., \cite[Propositions 4.2 and 4.6]{Conway90}),  $\sigma_{\rm{ess}}((A-i)^{-1})=\sigma_{\rm{ess}}(u(A-i)^{-1}u^*)=\sigma_{\rm{ess}}((B-i)^{-1})$. 
Therefore by Corollary \ref{cor: ess spec of resolvent}, $\sigma_{\rm{ess}}(A)=\sigma_{\rm{ess}}(B)$ and $A$ is bounded if and only if so is $B$.  
\end{proof}
Finally, Proposition \ref{prop: Matsuzawa relation is smooth} follows from the following Berg's generalization (see \cite[Theorem 39.8]{Conway} and \cite{Berg}) of Theorem \ref{thm: von Neumann theorem}, which is usually called the Weyl-von Neumann-Berg Theorem. 
\begin{theorem}[Weyl-von Neumann-Berg]\label{thm:Weyl-von Neumann-Berg}
Let $A,B$ be bounded normal operators on $H$. Then $\sigma_{\rm{ess}}(A)=\sigma_{\rm{ess}}(B)$ if and only if there exists $u\in \mathcal{U}(H)$ and $K\in \mathbb{K}(H)$ such that $uAu^*+K=B$. 
\end{theorem}
\begin{proof}[Proof of Proposition \ref{prop: Matsuzawa relation is smooth}]
We already know that $\overline{\sigma}_{\rm{ess}}$ is Borel. Therefore we have only to show that $AE_{\rm{u.c.res}}^{{\rm{SA}}(H)}B\Leftrightarrow \overline{\sigma}_{\rm{ess}}(A)=\overline{\sigma}_{\rm{ess}}(B)$ holds. $(\Rightarrow )$ holds by Lemma \ref{lem: A,B equiv then sigma bar are equal}. To show ($\Leftarrow $), assume that $\overline{\sigma}_{\rm{ess}}(A)=\overline{\sigma}_{\rm{ess}}(B)$ holds. If $A$, $B$ are bounded, then by Corollary \ref{cor: ess spec of resolvent}, we have $\sigma_{\rm{ess}}((A-i)^{-1})=\{(\lambda-i)^{-1};\lambda\in \sigma_{\rm{ess}}(A)\}=\sigma_{\rm{ess}}((B-i)^{-1})$, whence by Theorem \ref{thm:Weyl-von Neumann-Berg}, there exists $u\in \mathcal{U}(H)$ such that $(B-i)^{-1}-u(A-i)^{-1}u^*\in \mathbb{K}(H)$ holds. Hence $AE_{\rm{u.c.res}}^{{\rm{SA}}(H)}B$. If both $A$, $B$ are unbounded, then again by Corollary \ref{cor: ess spec of resolvent}, $\sigma_{\rm{ess}}((A-i)^{-1})=\sigma_{\rm{ess}}((B-i)^{-1})$ holds, whence $AE_{\rm{u.c.res}}^{{\rm{SA}}(H)}B$ by the same argument. 
\end{proof}
\begin{remark}
As can be expected from Corollary \ref{cor: ess spec of resolvent} and the proof of Proposition \ref{prop: Matsuzawa relation is smooth}, it is possible to introduce a Polish topology on the space Nor$(H)$ of normal operators on $H$ which extends the SRT on ${\rm{SA}}(H)$, such that ${\rm{Nor}}(H)\ni A\mapsto \sigma_{\rm{ess}}(A)\in \mathcal{F}(\mathbb{C})$ is Borel. Thus one can prove that ${\rm{SA}}(H)\ni A\mapsto \sigma_{\rm{ess}}((A-i)^{-1})\in \mathcal{F}(\mathbb{C})$ is a Borel reduction of $E_{\rm{u.c.res}}^{{\rm{SA}}(H)}$ to $\text{id}_{\mathcal{F}(\mathbb{C})}$. However, since this approach makes the arguments rather lengthy, we have decided to introduce $\overline{\sigma}_{\rm{ess}}$ instead. 
\end{remark}
\subsection{$E_{\mathbb{K}(H)_{\rm{sa}}}^{{\rm{SA}}(H)}$ is Generically Turbulent}
We have shown that $G\curvearrowright {\rm{SA}}(H)$ is not generically turbulent (Theorem \ref{thm: unbounded case: comeager orbit exists}), but if one restricts the action to its abelian subgroup $\mathbb{K}(H)_{\rm{sa}}$, it is generically turbulent: 
\begin{theorem}\label{thm: action of K(H) on SA(H) is turbulent}
The action of $\mathbb{K}(H)_{\rm{sa}}$ on ${\rm{SA}}(H)$ by addition is continuous, and 
$E_{\mathbb{K}(H)_{\rm{sa}}}^{{\rm{SA}}(H)}$ is a Borel equivalence relation. Moreover, the action is generically turbulent. 
\end{theorem}
The action is continuous by Proposition \ref{prop: B(H)_sa acts on SA(H) continuously}. Borelness of $E_{\mathbb{K}(H)_{\rm{sa}}}^{{\rm{SA}}(H)}$ then follows from the fact that the $\mathbb{K}(H)_{\rm{sa}}$-action is free (i.e., $A+K\neq A$ for every $K\neq 0$ and $A\in {\rm{SA}}(H)$) \cite[Proposition 3.4.8]{Gao09}.  
The rest of the proof is divided into steps, but compared to the proof of Theorem \ref{thm: action is weakly turbulent}, the arguments here are much simpler because for proving meagerness of orbits, there are many homeomorphic $\mathbb{K}(H)_{\rm{sa}}$-orbits and one can directly prove the turbulence of the $\mathbb{K}(H)_{\rm{sa}}$-action at $0\in {\rm{SA}}(H)$.
\begin{lemma}\label{lem: every orbit is dense}
For every $A\in {\rm{SA}}(H)$, the orbit $[A]_{\mathbb{K}(H)_{\rm{sa}}}$ is dense in ${\rm{SA}}(H)$.
\end{lemma}
\begin{proof}
By Weyl-von Neumann Theorem \ref{thm: Weyl-von Neumann}, there exists $K_0\in \mathbb{K}(H)_{\rm{sa}}$ such that $A+K_0$ is of the form $A+K_0=\sum_{n=1}^{\infty}a_ne_n$, where $a_n\in \mathbb{R}$ and $\{e_n\}_{n=1}^{\infty}$ is a sequence of mutually orthogonal rank one projections with sum equal to 1. Then the sequence $A_m:=(A+K_0)-\sum_{n=1}^ma_me_m\in [A]_{\mathbb{K}(H)_{\rm{sa}}}\ (m\in \mathbb{N})$ satisfies (by a similar argument to (\ref{eq: SRT approximation by unbounded operators}) in Lemma \ref{lem: B(H)_sa is standard})\ $A_m\stackrel{m\to \infty}{\to}0$ (SRT). 
Therefore by Lemma \ref{lem: Arai's lemma}, we have for each $K\in \mathbb{K}(H)_{\rm{sa}}$, $A_m+K\in [A]_{\mathbb{K}(H)_{\rm{sa}}}$ converges to $K$ in SRT. This shows that $\mathbb{K}(H)_{\rm{sa}}\subset \overline{[A]_{\mathbb{K}(H)_{\rm{sa}}}}$, and since $\mathbb{K}(H)_{\rm{sa}}$ is dense in ${\rm{SA}}(H)$, so is $[A]_{\mathbb{K}(H)_{\rm{sa}}}$.
\end{proof}

We next show that every orbit $[A]_{\mathbb{K}(H)_{\rm{sa}}}$ is meager. We first treat the case where $A$ is bounded.
\begin{lemma}\label{lem: orbit of bounded op is meager}
Let $B\in \mathbb{B}(H)_{\rm{sa}}$. Then its orbit $[B]_{\mathbb{K}(H)_{\rm{sa}}}$ is meager in ${\rm{SA}}(H)$.
\end{lemma}
\begin{proof}
Since $B$ is bounded, we have $[B]_{\mathbb{K}(H)_{\rm{sa}}}\subset \mathbb{B}(H)_{\rm{sa}}$. By Lemma \ref{lem: B(H)_sa is standard}, $\mathbb{B}(H)_{\rm{sa}}$ is meager in ${\rm{SA}}(H)$, whence so is $[B]_{\mathbb{K}(H)_{\rm{sa}}}$. 
\end{proof}
To prove the meagerness of $[A]_{\mathbb{K}(H)_{\rm{sa}}}$ for an unbounded $A$, we need easy lemmata. 
\begin{lemma}\label{lem: uncountable non-meager BPsets is impossible}
In a Polish space $X$, there is no uncountable disjoint family of non-meager subsets of $X$ each of which has the Baire property. 
\end{lemma}
\begin{proof}
Assume by contradiction that there exists an uncountable disjoint family of non-meager subsets $\{X_i\}_{i\in I}$ of $X$ such that each $X_i$ has the Baire property $(i\in I)$. Then for each $i\in I$, there exists a nonempty open subset $U_i$ of $X$ such that $U_i\setminus X_i$ is meager in $X$, (this is equivalent to that $U_i\setminus X_i$ is meager in $U_i$ with subspace topology, since $U_i$ is open). Since $\{U_i\}_{i\in I}$ is an uncountable family of nonempty open sets in a second countable space $X$, there exists $i_1,i_2\in I\ (i_1\neq i_2)$ such that $V:=U_{i_1}\cap U_{i_2}\neq \emptyset$. For $k=1,2$, $V\setminus X_{i_k}\subset U_{i_k}\setminus X_{i_k}$ is meager in $X$, whence $V\cap X_{i_k}$ is comeager in $V$. Therefore $(V\cap X_{i_1})\cap (V\cap X_{i_2})$ is also comeager in $V$. Since $V$ is open hence Baire, this shows in particular that $V\cap X_{i_1}\cap X_{i_2}\neq \emptyset$, which is a contradiction.    
\end{proof}
\begin{lemma}\label{lem: uncountably many homeomorphic orbits}
Let $A\in {\rm{SA}}(H)$ be unbounded. Then for each $s,t\in \mathbb{R}\setminus \{0\}$ with $s\neq t$, $[sA]_{\mathbb{K}(H)_{\rm{sa}}}$ and $[tA]_{\mathbb{K}(H)_{\rm{sa}}}$ are disjoint and homeomorphic.
\end{lemma}
\begin{proof}
Suppose by contradiction that $[sA]_{\mathbb{K}(H)_{\rm{sa}}}\cap [tA]_{\mathbb{K}(H)_{\rm{sa}}}\neq \emptyset$ for $s\neq t$. Then there exist $K_1,K_2\in \mathbb{K}(H)_{\rm{sa}}$ such that $sA+K_1=tA+K_2$. Therefore for $\xi \in \dom{A}$, $\frac{1}{t-s}(K_1-K_2)\xi=A\xi$. Since $\dom{A}$ is dense and $K_1-K_2$ is bounded, this implies that $A$ is also bounded, a contradiction. Therefore $[sA]_{\mathbb{K}(H)_{\rm{sa}}}\cap [tA]_{\mathbb{K}(H)_{\rm{sa}}}=\emptyset$. To show the latter claim, it is enough to show that for each $s\neq 0$, $[A]_{\mathbb{K}(H)_{\rm{sa}}}$ and $[sA]_{\mathbb{K}(H)_{\rm{sa}}}$ are homeomorphic. Define $\varphi:[A]_{\mathbb{K}(H)_{\rm{sa}}}\to [sA]_{\mathbb{K}(H)_{\rm{sa}}}$ by $\varphi(A+K):= sA+sK\ (K\in \mathbb{K}(H)_{\rm{sa}})$. It is straightforward to see that this is a well-defined homeomorphism. 
\end{proof}

\begin{proposition}\label{prop: meager orbit of unbounded A}
Let $A\in {\rm{SA}}(H)$ be unbounded. Then $[A]_{\mathbb{K}(H)_{\rm{sa}}}$ is meager in ${\rm{SA}}(H)$. 
\end{proposition}
\begin{proof}
Suppose that $[A]_{\mathbb{K}(H)_{\rm{sa}}}$ were non-meager. Since any orbit of a continuous action of a Polish group on a Polish is Borel, $[A]_{\mathbb{K}(H)_{\rm{sa}}}$ is a Borel subset of ${\rm{SA}}(H)$ by \cite[Proposition 3.1.10]{Gao09}. By Lemma \ref{lem: uncountably many homeomorphic orbits}, $\{[sA]_{\mathbb{K}(H)_{\rm{sa}}}\}_{s\in \mathbb{R}\setminus \{0\}}$ would be an uncountable family of disjoint Borel subsets of ${\rm{SA}}(H)$ any two of which are homeomorphic. Thus, each $[sA]_{\mathbb{K}(H)_{\rm{sa}}}\ (s\neq 0)$ would be non-meager and has the Baire property. This is a contradiction to Lemma \ref{lem: uncountable non-meager BPsets is impossible}. Therefore $[A]_{\mathbb{K}(H)_{\rm{sa}}}$ is meager.  
\end{proof}

\begin{proof}[Proof of Theorem \ref{thm: action of K(H) on SA(H) is turbulent}]
We have shown that every orbit is dense (Lemma \ref{lem: every orbit is dense}) and meager (Lemma \ref{lem: orbit of bounded op is meager} and Proposition \ref{prop: meager orbit of unbounded A}). Therefore to show the generic turbulence, it suffices to show that there exists at least one orbit on which the action is turbulent \cite[Proposition 8.7]{Kechris02}. We thus show that every local orbit of $0\in {\rm{SA}}(H)$ is somewhere-dense.  Let $U$ be an open neighborhood of 0 in ${\rm{SA}}(H)$, $V$ be an open neighborhood of 0 in $\mathbb{K}(H)_{\rm{sa}}$. We may and do assume that $U,V$ are of the following form
\eqa{
U&=\bigcap_{j=1}^m\{B\in {\rm{SA}}(H); \|(B-i)^{-1}\xi_j-(0-i)^{-1}\xi_j\|<\varepsilon\},\\
V&=\{K\in \mathbb{K}(H)_{\rm{sa}};\ \|K\|<\delta\}
}
for some unit vectors $\xi_1,\cdots,\xi_m\in H$ and $\varepsilon,\delta>0$. We show that $U\subset \overline{\mathcal{O}(0,U,V)}$. Let $B\in U$. By Weyl-von Neumann Theorem \ref{thm: Weyl-von Neumann} and spectral Theorem, there exists a sequence $\{B_n\}_{n=1}^{\infty}$ of finite-rank self-adjoint operators contained in $U$ such that $B_n\stackrel{\rm{SRT}}{\to}B$ in ${\rm{SA}}(H)$. Therefore to show that $B\in \overline{\mathcal{O}(0,U,V)}$, it suffices to prove that $B_n\in \mathcal{O}(0,U,V)$ for each $n\in \mathbb{N}$. Thus we may assume that $B$ is of finite-rank. Let $B=\sum_{k=1}^n\lambda_kp_k$ be the spectral decomposition of $B$. Choose $N\in \mathbb{N}$ so that $\frac{1}{N}\|B\|<\delta$. Then for each $1\le j\le m$ and $1\le l\le N$, 
\eqa{
\left \|\left (\frac{l}{N}B-i\right )^{-1}\xi_j-(0-i)^{-1}\xi_j\right \|^2&=
\sum_{k=1}^n\left |\frac{1}{\frac{l}{N}\lambda_k-i}-\frac{1}{0-i}\right |^2\|p_k\xi_j\|^2\\
&=\sum_{k=1}^n\frac{\lambda_k^2}{\lambda_k^2+\frac{N^2}{l^2}}\|p_j\xi_j\|^2\\
&\le \|(B-i)^{-1}\xi_j-(0-i)^{-1}\xi_j\|^2\\
&<\varepsilon^2.
}
Therefore $\frac{l}{N}B\in U$ for each $0\le l\le N$. Since $\frac{1}{N}B\in V$, this shows that $B=\frac{1}{N}B+\cdots +\frac{1}{N}B\in \mathcal{O}(0,U,V)$. Therefore $\mathcal{O}(0,U,V)$ is somewhere-dense, and the action is turbulent at 0. This finishes the proof.
\end{proof}

\subsection{$E_{\rm{dom},u}^{{\rm{SA}}(H)}$ is Borel}
Finally, we can also consider the unitary equivalence of domains of self-adjoint operators:
\begin{definition}
We define the domain unitary equivalence relation $E_{\rm{dom},u}^{{\rm{SA}}(H)}$ on ${\rm{SA}}(H)$ by 
$AE_{\rm{dom},u}^{{\rm{SA}}(H)}B$ if and only if there exists $u\in \mathcal{U}(H)$ such that $u\cdot \dom{A}=\dom{B}$.  
\end{definition}
This equivalence relation turns out to be Borel:
\begin{proposition}\label{prop: E_{dom,u} is Borel}
$E_{\rm{dom},u}^{{\rm{SA}}(H)}$ is a Borel equivalence relation. 
\end{proposition}
\begin{lemma}\label{lem: rank function is Borel}
Let $a,b\in \mathbb{R}, a<b$, and let $I=(a,b), [a,b)$ or $(a,b]$. Then the map ${\rm{SA}}(H)\ni A\mapsto {\rm{rank}}(E_A(I))\in \mathbb{N}\cup \{\infty\}$ is Borel. 
\end{lemma}
\begin{proof}
We show the case for $I=[a,b)$. Let $S_n:=\{A\in {\rm{SA}}(H);\ \text{rank}(E_A([a,b)))\le n\}\ (n\in \mathbb{N}\cup \{0\})$, $S_{\infty}:=\{A\in {\rm{SA}}(H);\ \text{rank}(E_A([a,b)))=\infty\}$. Then 
by a similar argument to the proof of  Proposition \ref{prop: essential spec contains lamda is dnese G_delta} (especially that $S_{n,k}$ defined there is SRT-closed), it can be shown that $S_n$ is SRT-closed. 
Therefore $\{A\in {\rm{SA}}(H); \text{rank}(E_A([a,b)))=n\}=S_n\setminus S_{n-1}\ (n\ge 1)$ and $S_0$ are Borel. Then $S_{\infty}={\rm{SA}}(H)\setminus \bigcup_{n\ge 0}S_n$ is Borel too. Thus the map $A\mapsto {\rm{rank}}(E_A(I))$ is Borel.
\end{proof}
\begin{proof}[Proof of Proposition \ref{prop: E_{dom,u} is Borel}]
It is easy to see that $\dom{A}=\dom{|A|+1}$ for every $A\in {\rm{SA}}(H)$, and $\dom{A}={\rm{Ran}}((|A|+1)^{-1})$. The associated subspace for $T_A=(|A|+1)^{-1}$ is 
\[H_n(T_A)=E_{T_A}((2^{-n-1},2^n])H,\ \ \ \ n\ge 0.\]
Note that for $\lambda \in \sigma(A)$, 
\[(|\lambda|+1)^{-1}\in (2^{-n-1},2^n]\Leftrightarrow \lambda \in \underbrace{(1-2^{n+1},1-2^n]\cup [2^n-1,2^{n+1}-1)}_{=:I_n\cup J_n}.\]
Let $d_0(A):=\text{rank}(E_A(1/2,1))$ and 
\[d_n(A):=\dim H_n(T_A)=\text{rank}(E_A(I_n))+\text{rank}(E_A(J_n)),\ \ \ \ \ n\ge 1.\]
By Lemma \ref{lem: rank function is Borel}, $d_n\colon {\rm{SA}}(H)\to \mathbb{N}\cup \{\infty\}$ is Borel for each $n\ge 0$.
 Define for $k,l,n\in \mathbb{N}\cup \{0\}$ a subset $\mathcal{B}_{k,l,n}$ of ${\rm{SA}}(H)\times {\rm{SA}}(H)$ by
 \[\mathcal{B}_{k,l,n}:=\left \{ (A,B); \sum_{i=0}^ld_{n+i}(A)\le \sum_{j=-k}^{l+k}d_{n+j}(B),\ 
\sum_{i=0}^ld_{n+i}(B)\le \sum_{j=-k}^{l+k}d_{n+j}(A)\right \}\]
Then each $\mathcal{B}_{k,l,n}$ is Borel. Therefore by Theorem \ref{thm: unitary equivalence of operator ranges}, 
\[E_{\rm{dom},u}^{{\rm{SA}}(H)}=\bigcup_{k\ge 0}\bigcap_{n\ge 0}\bigcap_{l\ge 0}\mathcal{B}_{k,l.n},\]
which is Borel.
\end{proof}
Note however that it is not clear whether the domain equivalence relation $AE_{\rm{dom}}^{{\rm{SA}}(H)}B\Leftrightarrow \dom{A}=\dom{B}$ is Borel or not. We show that it is co-analytic in Proposition \ref{prop: E_dom is coanalytic}.
\section{Concluding Remarks and Questions}\label{sec: Concluding remarks}
In this paper we have studied various equivalence relations on ${\rm{SA}}(H)$. 
There are also many other interesting equivalence relations involving the structure of self-adjoint operators. Let us finally state some equivalence relations and pose some related questions which we hope to come back in a future project. First of all we do not know if the Weyl-von Neumann equivalence relation is Borel. 
\begin{question}Is $E_G^{{\rm{SA}}(H)}$ Borel? 
\end{question}
Note that $E_G^{\mathbb{B}(H)_{\rm{sa}}}$ is Borel (because it is smooth) and that $E_{\mathbb{K}(H)_{\rm{sa}}}^{{\rm{SA}}(H)}$ is Borel (Theorem \ref{thm: action of K(H) on SA(H) is turbulent}). 
\begin{definition}
The {\it domain equivalence relation} $E_{\rm{dom}}^{{\rm{SA}}(H)}$ is the equivalence relation on ${\rm{SA}}(H)$ given by $AE_{\rm{dom}}^{{\rm{SA}}(H)}B$ if and only if $\dom{A}=\dom{B}$. 
\end{definition}
Although at the moment we do not know whether $E_{\rm{dom}}^{{\rm{SA}}(H)}$ is Borel or not, it is co-analytic: 
\begin{proposition}\label{prop: E_dom is coanalytic}
$E_{\rm{dom}}^{{\rm{SA}}(H)}$ is a co-analytic equivalence relation. 
\end{proposition}
\begin{lemma}\label{lem: xi in dom(A) is Borel}
The set $\{(A,\xi);\ \xi\in \dom{A}\}$ is a Borel subset of ${\rm{SA}}(H)\times H$.
\end{lemma}
\begin{proof}
Let $\xi\in H, A\in {\rm{SA}}(H)$. Then we have
\eqa{
\xi\in \dom{A}&\Leftrightarrow \lim_{t\to 0}\frac{e^{itA}\xi-\xi}{t}\ {\rm{exists}}\\
&\Leftrightarrow \forall \varepsilon\in \mathbb{Q}_+\ \exists \delta\in \mathbb{Q}_+\ {\rm{s.t.}}\ \forall s,t\in \mathbb{Q}^{\times}\cap (-\delta,\delta),\\
&\ \ \ \ \ \left \|\frac{e^{isA}\xi-\xi}{s}-\frac{e^{itA}\xi-\xi}{t}\right \|\le \varepsilon.
}
We note that 
\[F_{s,t}:=\left \{(A,\xi); \left \|\frac{e^{isA}\xi-\xi}{s}-\frac{e^{itA}\xi-\xi}{t}\right \|\le \varepsilon\right \}\]
is a closed subset of ${\rm{SA}}(H)\times H$. But this is clear, since the map ${\rm{SA}}(H)\times H\ni (A,\xi)\mapsto \frac{1}{t}(e^{itA}\xi-\xi)\in H$ is jointly continuous for each $t\in \mathbb{R}\setminus \{0\}$.
Therefore
\[\{(A,\xi);\ \xi\in \dom{A}\}=\bigcap_{\varepsilon \in \mathbb{Q}_+}\bigcup_{\delta \in \mathbb{Q}_+}\bigcap_{s,t\in \mathbb{Q}^{\times}\cap (-\delta,\delta)}F_{s,t}\]
is Borel. 
\end{proof}
\begin{proof}[Proof of Proposition \ref{prop: E_dom is coanalytic}]
Consider the following subsets $\mathcal{B},\mathcal{B}_1,\mathcal{B}_2$ of ${\rm{SA}}(H)\times {\rm{SA}}(H)\times H$, given by 
\eqa{
\mathcal{B}&:=\mathcal{B}_1\cup \mathcal{B}_2,\\
\mathcal{B}_1&:=\{(A,B,\xi);\ \xi\in \dom{A},\xi\notin \dom{B}\},\\
\mathcal{B}_2&:=\{(A,B,\xi);\ \xi\notin \dom{A},\xi\in \dom{B}\}.
}
Then it is clear that the complement of $E_{\rm{dom}}^{{\rm{SA}}(H)}$ is the projection of $\mathcal{B}$ onto the first two components ${\rm{SA}}(H)\times {\rm{SA}}(H)$. Therefore it suffices to prove that $\mathcal{B}$ is Borel. But
$\mathcal{B}_1=\{(A,B,\xi);\xi \in \dom{A}\}\cap \{(A,B,\xi);\xi \in \dom{B}\}^c$, which is Borel by Lemma \ref{lem: xi in dom(A) is Borel} and the definition of the product Borel structure. Similarly, $\mathcal{B}_2$ is Borel, and so is $\mathcal{B}$. Therefore $E_{\rm{dom}}^{{\rm{SA}}(H)}$ is co-analytic.
\end{proof}
\begin{question}
Is $E_{\rm{dom}}^{{\rm{SA}}(H)}$ Borel?
\end{question}
Note that since $E_{\rm{dom}}^{{\rm{SA}}(H)}$ is co-analytic, by Lusin Theorem it is enough to verify whether it is analytic or not. Before going to the next example, let us point out that:
\begin{proposition}\label{prop: Edom class is dense and meager}
Every $E_{\rm{dom}}^{{\rm{SA}}(H)}$-class is dense and meager.
\end{proposition}
\begin{proof}
Let $A\in {\rm{SA}}(H)$. Denote by $[A]$ the $E_{\rm{dom}}^{{\rm{SA}}(H)}$-class of $A$. We show that $[A]$ is dense. By Weyl-von Neumann Theorem \ref{thm: Weyl-von Neumann}, there exists $A_0\in [A]$ of the form $A_0=\sum_{n=1}^{\infty}a_ne_n$ where $\{a_n\}_{n=1}^{\infty}\subset \mathbb{R}$ and $\{e_n\}_{n=1}^{\infty}$ is a mutually orthogonal family or rank one projections with sum equal to 1. Let $B\in \mathbb{B}(H)_{\rm{sa}}$. Consider $A_N:=\sum_{n=N}^{\infty}a_ne_n+B$. Then $\dom{A_N}=\dom{A_0}=\dom{A}$, so $A_N\in [A]$. 
Moreover, $A_N\stackrel{N\to \infty}{\to}B$ (SRT). Therefore $\mathbb{B}(H)_{\rm{sa}}\subset [A]$. Therefore $[A]$ is dense in ${\rm{SA}}(H)$ because so is $\mathbb{B}(H)_{\rm{sa}}$.\\
We next show that $[A]$ is meager. If $A$ is bounded, then $[A]=\mathbb{B}(H)_{\rm{sa}}$, which is meager by Lemma \ref{lem: B(H)_sa is standard}. 
If $A$ is unbounded, then by Dixmier's Theorem (\cite{Dixmier49-1} and \cite[p.273, Corollary 1]{FillmoreWilliams}), there exists a strongly continuous one-parameter unitary group $\{u(t)\}_{t\in \mathbb{R}}$ such that $\dom{u(s)Au(s)^*}\cap \dom{u(t)Au(t)^*}=\{0\}$ for $s\neq t$. In particular, $[u(s)Au(s)^*]\neq [u(t)Au(t)^*]$ for $s\neq t$. Moreover, $[u(t)Au(t)^*]\ni B\mapsto u(s-t)Bu(s-t)^*\in [u(s)Au(s)^*]$ defines a homeomorphism for all $s,t\in \mathbb{R}$. Therefore if we show that each class $[A]$ has the Baire property, they must be meager by Lemma \ref{lem: uncountable non-meager BPsets is impossible}.\\ \\
\textbf{Claim}. $[A]$ is co-analytic, whence it has the Baire property.\\
 Since $\dom{A}$ is an operator range, it is $F_{\sigma}$. In particular it is Borel. Moreover, $\{(\xi,B);\ \xi \in \dom{B}\}$ is Borel in $H\times {\rm{SA}}(H)$ by Lemma \ref{lem: xi in dom(A) is Borel}. 
Therefore 
\eqa{
\mathcal{S}_1&:=\{(\xi,B);\xi\in \dom{B},\xi\notin \dom{A}\}\\
&=\dom{A}^c\times {\rm{SA}}(H)\cap \{(\xi,B);\xi \in \dom{B}\}
}
is Borel. 
Similarly, $\mathcal{S}_2:=\{(\xi,B);\xi\notin \dom{B},\xi \in \dom{A}\}$ 
is Borel, and so is $\mathcal{S}:=\mathcal{S}_1\cup \mathcal{S}_2$.\\
Now for $B\in {\rm{SA}}(H)$,
\eqa{\dom{A}\neq \dom{B}&\Leftrightarrow \exists \xi \ \ [(\xi\in \dom{A}\wedge \xi\notin \dom{B})\vee (\xi \notin \dom{A}\wedge \xi\in \dom{B})]\\
&\Leftrightarrow \exists \xi \ \ [(\xi,B)\in \mathcal{S}],
}
so $[A]^c$, which is the projection of $\mathcal{S}$ onto the second component, is analytic. This finishes the proof.
\end{proof}

Note that if one replaces ``equality of domains" by ``unitary equivalence of domains", it is indeed a Borel equivalence relation (see Proposition \ref{prop: E_{dom,u} is Borel}).

One can also consider an equivalence relation coming from relatively compact perturbations and unitary equivalence. Recall that a symmetric operator $K$ is {\it relatively $A$-compact} for $A\in {\rm{SA}}(H)$, if $\dom{K}\supset \dom{A}$ and $K|_{\dom{A}}$ is compact as a map $(\dom{A},\|\cdot \|_A)\to H$, where $\|\cdot \|_A$ is the graph norm of $A$. In this case, $A+K$ is self-adjoint (see e.g. \cite[$\S$8]{Schmudgen} for details). 
\begin{definition}
The equivalence relation $E_{\rm{u,rel.c}}^{{\rm{SA}}(H)}$ on ${\rm{SA}}(H)$ is given by $AE_{\rm{u,rel.c}}^{{\rm{SA}}(H)}B$ if and only if there exists $u\in \mathcal{U}(H)$ and $K\in {\rm{SA}}(H)$ which is relatively $A$-compact, such that $u(A+K)u^*=B$.  
\end{definition}
It is easy to see that $E_{\rm{u,rel.c}}^{{\rm{SA}}(H)}$ is an equivalence relation. It seems to be a more appropriate criterion of classifying unbounded self-adjoint operators than $E_G^{{\rm{SA}}(H)}$, in view of the next observation:
\begin{example}[Example \ref{ex: unitarily quivalent domains but not uAu*+K=B} continued]
 Consider $B_s,B_{t}\ (s,t\in [0,1],s\neq t)$ from Example \ref{ex: unitarily quivalent domains but not uAu*+K=B}. 
We have seen that $\dom{B_s}=\dom{B_t}$ and $\sigma_{\rm{ess}}(B_s)=\sigma_{\rm{ess}}(B_t)=\mathbb{N}$, but $(B_s,B_t)\notin E_G^{{\rm{SA}}(H)}$.  
 On the other hand, $B_sE_{\rm{u,rel.c}}^{{\rm{SA}}(H)}B_t$ holds:  indeed, let
\[K_0:=-\sum_{k,m=1}^{\infty}\frac{s-t}{m+2}e_{\nai{k}{m}}.\]
Then $K_0$ is not compact (since $\sigma_{\rm{ess}}(K_0)=\{0\}\cup \{-\frac{s-t}{m+2}|\ m\in \mathbb{N}\}\neq \{0\}$), but it is actually relatively $B_s$-compact, and $B_s+K_0=B_t$ holds. To see this, it suffices to show that $K_0(B_s-i)^{-1}$ is compact (cf. \cite[Proposition 8.14]{Schmudgen}). Note that since 
\[K_0(B_s-i)^{-1}=-\sum_{k,m=1}^{\infty}\frac{s-t}{(m+2)(k+\frac{s}{m+2}-i)}e_{\nai{k}{m}},\]
the spectra $\sigma(K_0(B_s-i)^{-1})$ is the closure $\overline{\mathcal{M}}$ of $\mathcal{M}$,
\[\mathcal{M}:=\left \{\gamma_{k,m}:=\frac{t-s}{(m+2)(k+\frac{s}{m+2}-i)};\ k,m\in \mathbb{N}\right \}.\]
It is easy to see that the only accumulation point of $\overline{\mathcal{M}}$ is $\{0\}$. Therefore $K_0(B_s-i)^{-1}$ is compact, so $K_0$ is relatively $B_s$-compact. Therefore, the family $\{B_t\}_{t\in [0,1]}$ belongs to a single $E_{\rm{u,rel.c}}^{{\rm{SA}}(H)}$-class. 
\end{example}
At present stage we do not know the answers to:
\begin{question}
Is $E_{\rm{u,rel.c}}^{{\rm{SA}}(H)}$ Borel? Is it Borel reducible to $E_{G'}^Y$ for some Polish group $G'$ and a Polish $G'$-space $Y$? 
\end{question}
Finally, let us remark that different way of perturbing self-adjoint operators may give rise to distinct equivalence relations. 
\begin{definition}
Let $1\le p<\infty$, and let $S^p(H)_{\rm{sa}}$ be the additive Polish group of self-adjoint Schatten $p$-class operators on $H$  equipped with Schatten $p$-norm. $S^p(H)_{\rm{sa}}$ acts on ${\rm{SA}}(H)$ by addition , and we may consider an action of $G_p:=S^p(H)_{\rm{sa}}\rtimes \mathcal{U}(H)$ on ${\rm{SA}}(H)$ analogous to $G\curvearrowright {\rm{SA}}(H)$. 
\end{definition}
It is especially of interest to know whether one of $E_{G_1}^{{\rm{SA}}(H)}$ and $E_G^{\rm{SA}}(H)$ is Borel reducible to the other (note that by Kato-Rosenblum Theorem \cite{Kato57,Rosenblum}, trace class perturbation is rather different from other Schatten class or compact perturbations). Note also that the orbit equivalence relation of $S^p(H)_{\rm{sa}}$-action on ${\rm{SA}}(H)$ can be thought of as a non-commutative version of the $\ell^p$-action on $\mathbb{R}^{\mathbb{N}}$ studied by Dougherty-Hjorth \cite{DoughertyHjorth99}.   
\section*{Acknowledgments}
We would like to thank Professor Asao Arai for communicating us his proof of Lemma \ref{lem: Arai's lemma}, Professors Alexander Kechris and  Asger T\"ornquist  for useful comments regarding turbulence argument and pointing us to the literature.   
HA is supported by EPDI/JSPS/IH\'ES Fellowship. He also thanks Kyoto University GCOE program from which he received financial support for traveling to attend the workshop ``Set theory and C$^*$-algebras" held in American Institute of Mathematics in 2012, where he learned the basic ideas of descriptive set theory and obtained the motivation for the current project.  

Hiroshi Ando\\
Erwin Schr\"odinger International Institute for Mathematical Physics,\\
2. Stock, Boltzmanngasse 9\\
1090 Wien Austria\\
hiroshi.ando@univie.ac.at\\@\\
New contact address (from March 1 2014):\\
Department of Mathematical Sciences,\\
University of Copenhagen\\
Universitetsparken 5\\
2100 Copenhagen \O \ Denmark\\
http://andonuts.miraiserver.com/index.html\\ \\
Yasumichi Matsuzawa\\
Department of Mathematics, Faculty of Education, Shinshu University\\
6-Ro, Nishi-nagano, Nagano, 380-8544, Japan\\
myasu@shinshu-u.ac.jp\\
https://sites.google.com/site/yasumichimatsuzawa/home
\end{document}